\documentclass[a4paper,12pt,reqno]{amsart}

\usepackage[T1]{fontenc}
\usepackage[utf8]{inputenc} 
\usepackage[english]{babel}
\usepackage{enumerate}

\usepackage{amssymb}

\numberwithin{equation}{section}

\usepackage{tikz}
\usetikzlibrary{arrows}
\pgfarrowsdeclare{<<<}{>>>}
{
\arrowsize=0.2pt
\advance\arrowsize by .5\pgflinewidth
\pgfarrowsleftextend{-4\arrowsize-.5\pgflinewidth}
\pgfarrowsrightextend{.5\pgflinewidth}
}
{
\arrowsize=0.2pt
\advance\arrowsize by .5\pgflinewidth
\pgfpathmoveto{\pgfpointorigin}
\pgfpathlineto{\pgfpoint{-2mm}{-1.3mm}}
\pgfusepathqstroke
\pgfpathmoveto{\pgfpointorigin}
\pgfpathlineto{\pgfpoint{-2mm}{1.3mm}}
\pgfusepathqstroke
}

\pgfarrowsdeclare{<<}{>>}
{
\arrowsize=0.2pt
\advance\arrowsize by .5\pgflinewidth
\pgfarrowsleftextend{-4\arrowsize-.5\pgflinewidth}
\pgfarrowsrightextend{.5\pgflinewidth}
}
{
\arrowsize=0.2pt
\advance\arrowsize by .5\pgflinewidth
\pgfpathmoveto{\pgfpointorigin}
\pgfpathlineto{\pgfpoint{-1.2mm}{-.8mm}}
\pgfusepathqstroke
\pgfpathmoveto{\pgfpointorigin}
\pgfpathlineto{\pgfpoint{-1.2mm}{.8mm}}
\pgfusepathqstroke
}

\usepackage{caption}
\def\N{\mathbb{N}}
\def\R{\mathbb{R}}
\def\S{\mathbb{S}}
\def\P{\mathbb{P}}
\def\Ceeb{C_{\rm per}}
\def\Cgr{C_{\rm vol}}
\def\proofof#1{\begin{proof}[Proof of #1]}
\def\angle#1#2#3{#1\widehat{#2}#3}
\def\comp{\subset\subset}
\def\eps{\varepsilon}
\def\doppio#1#2{\stackrel{\hbox{\tiny $#1$}}{\hbox{\tiny $#2$}}}
\def\step#1#2{\par\vspace{5pt}\noindent{\underline{\it Step~#1.}}\emph{ #2}\\}

\def\E{\mathcal{E}}
\def\haus{\mathcal{H}}
\def\F{\mathcal F}

\theoremstyle{plain}
\newtheorem{thm}{Theorem}[section]
\newtheorem{lem}[thm]{Lemma}
\newtheorem{prop}[thm]{Proposition}
\newtheorem{cor}[thm]{Corollary}
\newtheorem{defn}[thm]{Definition}
\newtheorem{rmk}[thm]{Remark}

\usepackage[textsize=tiny]{todonotes}
\allowdisplaybreaks

\usepackage{hyperref}
\usepackage{bookmark}

\title[Steiner property for planar clusters. Isotropic case.]{On the Steiner property for planar minimizing clusters. The isotropic case.}

\author[V. Franceschi]{Valentina Franceschi$^*$}
\address{$^*$Department of Mathematics, University of Padova}
\email{valentina.franceschi@unipd.it}

\author[A. Pratelli]{Aldo Pratelli$^\flat$}
\address{$^\flat$Department of Mathematics, University of Pisa}
\email{aldo.pratelli@unipi.it}

\author[G. Stefani]{Giorgio Stefani$^\sharp$}
\address{$^\sharp$Scuola Internazionale Superiore di Studi Avanzati Trieste}
\email{gstefani@sissa.it}

\begin{document}

\begin{abstract}
We consider the isoperimetric problem for clusters in the plane with a double density, that is, perimeter and volume depend on two weights. In this paper we consider the isotropic case, in the parallel paper~\cite{FPS2} the anisotropic case is studied. Here we prove that, in a wide generality, minimal clusters enjoy the ``Steiner property'', which means that the boundaries are made by ${\rm C}^{1,\gamma}$ regular arcs, meeting in finitely many triple points with the $120^\circ$ property.
\end{abstract}

\maketitle

\section{Introduction}

In this paper we consider the isoperimetric problem with double density for planar clusters. This means that we are given two l.s.c. functions $g,\,h:\R^2\to \R^+$ (the \emph{densities}), and the volume and perimeter of any set $E\subseteq\R^2$ of locally finite perimeter are given by
\begin{align}\label{weightedvolper}
|E| = \int_E g(x)\,dx\,, && P(E) = \int_{\partial^* E} h(x)\, d\haus^1(x)\,,
\end{align}
where $\partial^* E$ denotes the reduced boundary as usual in this context (see~\cite{AFP} for definitions and properties of sets of finite perimeter). The \emph{isoperimetric problem} consists then, as always, in the search of sets of given (weighted) volume which minimize the (weighted) perimeter. Of course, depending on $g$ and $h$, different situations may occur. This generalization of the standard isoperimetric problem has gained a rapidly increasing interest in the last decades due to the work of several authors, see for instance~\cite{RCBM08,CMV10,CRS12,MP13,BBCLT16,ABCMP17,DFP,G18,ABCMP19,CGPRS20,FP20} and the references therein.\par

The \emph{isoperimetric problem for clusters}, instead, consists in minimizing the total pe\-ri\-me\-ter of a union of sets with given volume. More precisely, for a given $m\in \N$, an \emph{$m$-cluster} is a collection $\E=\{E_1,\,E_2,\, \dots\, ,\, E_m\}$ of $m$ essentially disjoint sets of locally finite perimeter. For brevity, we will denote $E_0=\R^2\setminus (\cup_{i=1}^m E_i)$. The \emph{volume} of a cluster $\E$ is the vector $|\E| = (|E_1|,\, |E_2|,\, \dots\,,\, |E_m|)\in (\R^+)^m$, while its \emph{perimeter} is
\[
P(E) = \frac{P(\cup_{i=1}^m E_i) + \sum_{i=1}^m P(E_i) }2 
= \frac{\sum_{i=0}^m P(E_i) }2 
= \int_{\partial^*\E} h(x)\, d\haus^1(x)\,,
\]
where the ``reduced boundary'' $\partial^* \E$ of a cluster is defined as the union of the reduced boundaries, that is,
\[
\partial^*\E = \cup_{i=1}^m \partial^* E_i\,.
\]
A cluster which minimizes the perimeter among all those with fixed volume is usually called \emph{minimal cluster}. The isoperimetric problem for clusters has been deeply studied in the last decades. In particular, for the Euclidean case, corresponding to $g\equiv h\equiv 1$, the problem is often referred to as the ``double bubble problem'' in the case $m=2$, which was completely solved in~\cite{HMRR,RHYS03,Rei}. 
Also the ``triple bubble'' and the ``quadruple bubble'', corresponding to $m=3$ and $m=4$, have been studied, see~\cite{W} and~\cite{PT0,PT1,PT2} respectively.\par

The aim of this paper is to prove that minimal clusters enjoy the ``Steiner property'', that is, their boundaries are composed by finitely many ${\rm C}^{1,\gamma}$ arcs, which meet in finitely many triple points with the $120^\circ$ property, see Definition~\ref{SteinerProperty} below. While this property is widely known for the Euclidean case, see for instance the classical papers~\cite{Tay,M94} or the recent book~\cite{Mag}, we generalize it to a much more general case. Similar questions are considered also in~\cite{M94} where the author studies existence and regularity results for minimal clusters in compact Riemannian surfaces. In particular, we will prove the validity of the Steiner property as soon as an $\eps-\eps^\beta$ property and a volume growth condition hold. 

\begin{defn}[$\eta$-growth condition and $\eps-\eps^\beta$ property for clusters]\label{defge}
Given a power $\eta\ge 1$, an \emph{$\eta$-growth condition} is said to hold if there exist a large constant $\Cgr$ and a small constant $R_\eta$ such that, for every $x\in\R^2$ and every $r<R_\eta$, the ball $B(x,r)$ has volume $|B(x,r)|\leq \Cgr r^\eta$. We say that the \emph{local $\eta$-growth condition} holds if for any bounded domain $D\comp \R^2$ there exist two constants $\Cgr$ and $R_\eta$ such that the above property holds for balls $B(x,r)\subseteq D$.\par
Moreover, we say that a cluster $\E$ satisfies the $\eps-\eps^\beta$ property for some $0<\beta\leq 1$ if there exist two small constants $\bar\eps,\,R_\beta$ and a large constant $\Ceeb$ such that, for every vector $\eps\in \R^m$ with $|\eps|\leq \bar\eps$ and every $x\in\R^2$, there exists another cluster $\F$ such that
\begin{align}\label{propeeb}
\F\Delta \E \comp \R^2\setminus B(x,R_\beta)\,, &&
|\F|=|\E|+\eps\,, &&
P(\F)\leq P(\E) + \Ceeb |\eps|^\beta\,.
\end{align}
In this case, for every $0<t\leq \bar\eps$ we call $\Ceeb[t]$ the smallest constant such that the above property is true whenever $|\eps|\leq t$. The map $t\mapsto \Ceeb[t]$ is clearly increasing and $\Ceeb[\bar\eps]\leq \Ceeb$.
\end{defn}

We underline that both the above assumptions are satisfied for a wide class of densities and clusters. In particular, the growth (or local growth) condition clearly holds with $\eta=2$ whenever the density $g$ is bounded (or locally bounded). Concerning the $\eps-\eps^\beta$ property, this is a crucial tool when dealing with isoperimetric problems. It is simple to observe that it is valid with $\beta=1$ for every cluster of locally finite perimeter whenever the density $h$ is regular enough (at least Lipschitz). It is also known that, if $h$ is $\alpha$-H\"older for some $0<\alpha\leq 1$, then every cluster of locally finite perimeter satisfies the $\eps-\eps^\beta$ property with
\[
\beta=\frac 1{2-\alpha}\,,
\]
the proof can be found in~\cite{CP} for the special case $m=1,\, g=h$, in~\cite{PS19} for the more general case $m=1$, and in~\cite{PS2} for the fully general case. 
Relaxing even further the continuity assumptions to consider the ``degenerate'' case $\alpha=0$, the $\eps-\eps^{1/2}$ property holds as soon as $h$ is locally bounded away from $0$ and $+\infty$ (that is, in every compact set one has $1/M\leq h\leq M$ for a suitable $M$). If in addition $h$ is continuous, instead, not only the $\eps-\eps^{1/2}$ property holds, but also $\lim_{t\to 0} \Ceeb[t]=0$. We can be even more precise, that is, $\Ceeb[t]\lesssim \sqrt{\omega_h(\sqrt t)}$, where $\omega_h$ is the modulus of continuity of $h$ (see~\cite{PS2}).\par

We can now give the formal definitions of the Steiner property, already described above, and of the Dini property.

\begin{defn}[Steiner property]\label{SteinerProperty}
A cluster $\E$ is said to satisfy the \emph{Steiner property} if $\partial^* \E$ is a locally finite union of ${\rm C}^1$ arcs, and at each junction point exactly three arcs are meeting, having tangent vectors which form three angles of $\frac 23\,\pi$.
\end{defn}

\begin{defn}[Dini property]
Let $\varphi:\R^+\to\R^+$ be an increasing function such that $\varphi(0)=0$. We say that \emph{$\varphi$ satisfies the Dini property} if for every $C>1$ one has
\[
\sum_{n\in\N} \varphi(C^{-n}) < +\infty\,,
\]
and we say that \emph{$\varphi$ satisfies the $1/2$-Dini property} if $\sqrt{\varphi}$ satisfies the Dini property. A uniformly continuous function $f:\R^N\to\R^+$ is Dini continuous if and only if its modulus of continuity $\omega_f$ satisfies the Dini property. We say that a function $f:\R^N\to\R$ is \emph{locally Dini continuous} if for every bounded set $\Omega\comp\R^N$ the modulus of continuity of $f$ within $\Omega$ satisfies the Dini property. Moreover, we say that \emph{$f$ is $1/2$-Dini continuous} if $\omega_f$ satisfies the $1/2$-Dini property, and that \emph{$f$ is locally $1/2$-Dini continuous} if the $1/2$-Dini property is satisfied by the modulus of continuity of $f$ in every bounded subset of $\R^N$.
\end{defn}

Notice that if $\varphi:\R^+\to\R^+$ is non-decreasing, such that $\varphi(0)=0$ and $\alpha$-H\"older, then it satisfies the Dini property. We are now in position to state the main result of the present paper.

\begin{thm}[Steiner regularity for minimal clusters]\label{main}
Let $h$ be locally $1/2$-Dini continuous, let $\E$ be a minimal cluster, and let us assume that for some $\eta,\,\beta$ the local $\eta$-growth condition holds, as well as the $\eps-\eps^\beta$ property for $\E$. Assume in addition that either
\begin{enumerate}[(i)]
\item $\eta \beta>1$, or
\item $\eta\beta=1$ and the function $t\mapsto \Ceeb[t]$ satisfies the $1/2$-Dini property.
\end{enumerate}
Then $\E$ satisfies the Steiner property, and if $\eta\beta>1$ and $h$ is locally $\alpha$-H\"older then the arcs of $\partial^*\E$ are actually ${\rm C}^{1,\gamma}$ with $\gamma=\frac 12\, \min\left\{\eta\beta-1,\, \alpha\right\}$.
\end{thm}

Notice that, as said above, if $h$ is locally $\alpha$-H\"older then the $\eps-\eps^\beta$ property is surely true with $\beta=(2-\alpha)^{-1}$. But of course, it could be true also with some larger $\beta$.

A few comments are now in order. First of all we observe that, in the classical Euclidean case, one has $\beta=1$ and $\eta=2$, hence the assumptions of our result cover an extremely more general case than the classical one.

Concerning the case $\eta\beta=1$, in order to get the Steiner property of $\E$ we have added a Dini-type property on $\Ceeb$, while if $\eta\beta>1$ no additional assumption was needed. 
In fact, the ${\rm C}^1$ regularity of the boundary fails if $\eta\beta=1$ without extra assumptions, already fails if $m=1$. For instance, let us consider the densities given by $g\equiv 1$ in the whole $\R^2$, while $h\equiv 1$ on a closed square of unit volume in $\R^2$, and $h\equiv 3$ outside. It is easy to see that the only minimizer with unit volume is the square $Q$ itself, which is not ${\rm C}^1$. In this case, the $\eps-\eps^{1/2}$ property holds true since $h$ is bounded away from $0$ and $+\infty$, and the $\eta$-growth condition is satisfied with $\eta=2$.

On the bright side, Theorem~\ref{main} can always be applied if $g$ is locally bounded and $h$ is $1/4$-Dini continuous (i.e., $\sqrt[4]{\omega_h}$ satisfies the Dini property). Indeed, in this case $\eta=2$ and $\beta=1/2$, and the required continuity of $h$ and $\Ceeb$ follows by the fact that $\Ceeb[t]\lesssim \sqrt{\omega_h(\sqrt{t})}$
, already observed above. The fact that, in order to get ${\rm C}^1$ regularity of the boundary, some $1/2$-Dini property is needed, is standard, see for instance~\cite{Tam,Mag}.\par

It is to be remarked that the boundedness of an optimal cluster is false in general, but true under quite mild assumptions. The boundedness of isoperimetric sets, or optimal clusters, is a well studied question, also because it is deeply connected with the existence, see for instance~\cite{CP,DFP,PS18,PS19}. Of course, whenever optimal clusters are a priori known to be bounded, as soon as Theorem~\ref{main} applies then they are made by a finite (and not just locally finite) union of regular arcs.

We conclude this introduction by pointing out that, in the isoperimetric problem with double density, one can consider an anisotropic density for the perimeter, that is, the density $h$ may also depend on the direction of the normal vector. In other words, $h$ is defined on $\R^2\times \S^1$, and the term $h(x)$ in the right definition in~(\ref{weightedvolper}) is replaced by $h(x,\nu_E(x))$, where $\nu_E(x)$ is the unit normal vector to $\partial^*E$ at $x\in\partial^* E$. The anisotropic case is of course more complicate to treat, but it is a very important generalization, in particular it is quite natural when considering manifolds with density (see for instance~\cite[Chapter 18]{Morganbook}). We are able to study the Steiner property also in the general anisotropic case, to which the parallel paper~\cite{FPS2} is devoted. An important peculiarity of the isotropic case is the $120^\circ$ property, which is in general false in the anisotropic case.

\subsection{Plan of the paper}

The proof of Theorem~\ref{main} is presented in Section~\ref{sec:proof}. To help the reader, we present a quick scheme of the construction in Section~\ref{s:scheme} below. Then, in Section~\ref{s:examples} we apply Theorem~\ref{main} to specific examples. In particular, in Section~\ref{ssgr} we consider the so-called Grushin perimeter, establishing the validity of (a suitable reinterpretation of) the Steiner property, that was suggested in~\cite{FS}. The characterisation of minimal double bubbles in this context remains an open problem. In addition, in Section~\ref{ssgauss} we consider the case of two equal Gaussian densities, proving the Steiner property and the smoothness of the free boundary of minimal clusters.

\subsection{Scheme of the proof}\label{s:scheme}

The core of the proof is to show that, under the assumptions of Theorem~\ref{main}, minimal clusters have (locally) finitely many triple junctions at 120 degrees.
Once this is proved, the argument to obtain either the ${\rm C}^1$ or the ${\rm C}^{1,\gamma}$ regularity of the free boundary is standard. In particular, if $\eta\beta=1$ it also relies on the $1/2$-Dini continuity of $h$ and of $t\mapsto \Ceeb[t]$, see Section~\ref{s:interface}. 

In order to show that multiple junctions of minimal clusters are locally finite (and triple), we start by observing two facts. First of all, a simple geometric argument, that we present in Lemma~\ref{steiner}, allows us to obtain a quantitative description of the ``$120^\circ$ rule'' in the plane by parametrizing the (Euclidean) length of the Steiner tree between three points in terms of an angle between them. 
The second observation is contained in Lemma~\ref{13/2} where we show that the (Euclidean) length of the (reduced) boundary of a minimal cluster inside a small ball is controlled by a constant multiple of its radius. This is a simple consequence of the validity of the $\varepsilon-\varepsilon^\beta$ property and of the $\eta$-growth condition, summarised together as in Lemma~\ref{labello}.

Following the two observations above, consider a minimal cluster $\mathcal E=\{E_1,\dots,E_m\}$ and call ``colors'' the regions $E_0,E_1,\dots,E_m$. We want to show that the junctions of $\partial\mathcal E$ consist of triple points only and that they are a positive distance apart one from another. To this purpose,
in Lemma~\ref{3points} we show that $\partial^*\mathcal E$ intersects the boundary of a ball $B(x,r)$ in at most $3$ points for many small radii $r$. Moreover, the ``no-island'' Lemma~\ref{noisland} ensures us that there are no colors $E_i$, $i=0,\dots,m$ having positive measure inside a small ball $B(x,r)$, 
without containing a non-trivial part of its boundary, that is, with positive $\mathcal H^1$ measure. 
These two facts together then allow to deduce that the colors $E_i$ having positive measure inside a sufficiently small ball are at most three.

Calling $x\in\R^2$ a $3$-color point if (at least) three different ``colors'' of $\mathcal E$ intersect $B(x,r)$ for all $r>0$, the above analysis allows us to infer that at (many) small scales any $3$-color point has exactly $3$ boundary points on $\partial B(x,r)$, see Lemma~\ref{9/10}. This is a key ingredient in the proof of Proposition~\ref{dist>R6}, where we conclude by showing that any two $3$-color points of $\mathcal E$ are indeed a positive distance apart.

To conclude we are only left to show that $3$-color points are indeed triple junctions. This will follow, once the regularity of the free boundary is proved as in Section~\ref{s:interface}, by observing that at $3$-color points, only three regular arcs of $\partial E$ are meeting. 

We emphasize that the assumption that $h$ is positive and locally bounded and the validity of the $\varepsilon-\varepsilon^\beta$ property and of the $\eta$-growth condition are used along the entire argument. Instead, the $1/2$-Dini continuity of $h$ (and of $t\mapsto \Ceeb[t]$) is only used to prove the regularity of the free boundary, and thus multiple points must be finitely many and triple even removing this assumption. Moreover, while the simple, quantitative Steiner-type estimate of Lemma~\ref{steiner} is at the basis of the proof in the isotropic case, the anisotropic situation needs more delicate arguments (see \cite[Section~2.2]{FPS2}).

\section{Proof of the main result\label{sec:proof}}

The proof of the main result, Theorem~\ref{main}, is presented in this section. In turn, this is subdivided in four subsections. The first one collects some basic definitions and technical tools, in the second one we show that there are finitely many junction points, each of which where exactly three different sets meet, and in the third one we obtain the regularity. The actual proof of the theorem, presented in the last subsection, basically only consists in putting the different parts together.

Since we aim to prove Theorem~\ref{main}, from now on we assume that $h$ is locally $1/2$-Dini continuous and that the local $\eta$-growth condition holds for some $\eta\ge 1$.

\subsection{Some definitions and technical tools}

Let us fix some notation, that will be used through the rest of the paper. Since we are interested in a local property, in the proof of Theorem~\ref{main} we will immediately start by fixing a big closed ball $D\subseteq\R^2$, and the whole construction will be performed there. Hence, all the following definitions will depend upon $D$, in particular we assume that $|B(x,r)|\leq \Cgr r^\eta$ for every ball $B(x,r)\subseteq D$.\par

Since $h$ is locally $1/2$-Dini continuous, hence in particular continuous, we can call $\omega:\R^+\to\R^+$ its modulus of continuity inside $D$. In particular, if $h$ is locally $\alpha$-H\"older, then $\omega(r) \leq C r^\alpha$ for a suitable constant $C$. Moreover, we will call $0<h_{\rm min}\leq h_{\rm max}$ the maximum and the minimum of $h$ in $D$. Keep in mind that, if $\eta\beta=1$, then the function $t\mapsto \Ceeb[t]$ is assumed to satisfy the $1/2$-Dini property, so in particular it is infinitesimal. As a consequence, we can fix $\bar\eps$ so small that $\Ceeb$ is as small as desired, in particular we will use several times that
\begin{equation}\label{Cperissmall}
\Ceeb=\Ceeb[\bar\eps] \ll \frac{h_{\rm min}^2}{\Cgr^\beta\,h_{\rm max}}\,.
\end{equation}
We can easily observe now a simple estimate between volume and perimeter of any set $E$. 

\begin{lem}[Isoperimetric inequality with exponent]\label{disgusting}
For every set $E\subseteq D$ we have
\[
P(E) \geq \frac {h_{\rm min}}{\Cgr^{1/\eta}}\, |E|^{1/\eta}\,.
\]
\end{lem}
\begin{proof}
By approximation, we can limit ourselves to consider the case of a planar, polygonal set $E$. A classical result from Gustin (see~\cite{Gustin}) says that such a set can be covered with countably many balls $B_i=B(x_i,r_i)$ in such a way that
\[
\haus^1(\partial^* E) \geq 2\sqrt 2 \sum\nolimits_i r_i\,.
\]
Keeping in mind that $\eta\geq 1$, for any $E\subseteq D$ we deduce
\[\begin{split}
P(E)&\geq h_{\rm min} \haus^1(\partial^* E)
\geq 2\sqrt 2 h_{\rm min} \sum\nolimits_i r_i
\geq \frac{2\sqrt 2 h_{\rm min}}{\Cgr^{1/\eta}}\, \sum\nolimits_i |B(x_i,r_i)|^{1/\eta}\\
&\geq \frac{2\sqrt 2 h_{\rm min}}{\Cgr^{1/\eta}}\, \Big(\sum\nolimits_i |B(x_i,r_i)|\Big)^{1/\eta}
\geq \frac{2\sqrt 2 h_{\rm min}}{\Cgr^{1/\eta}}\, |E|^{1/\eta}\geq \frac{h_{\rm min}}{\Cgr^{1/\eta}}\, |E|^{1/\eta}\,,
\end{split}\]
so the proof is concluded.
\end{proof}

The following is a simple geometric fact. This is specific for the isotropic case. The analogous property in the anisotropic case is weaker and much more complicate to obtain.

\begin{lem}[The $120^\circ$ net property]\label{steiner}
There exists a continuous and strictly increasing function $L:[0,2\pi/3]\to [1,2]$ such that, if $x,\,y,\,z$ are three points in $\R^2$ such that $|y-x|=|z-x|$, and $\theta=\angle zxy\in [0,2\pi/3]$, then there exists a connected set $\Gamma$ contained in the triangle $xyz$, containing $x,\,y$ and $z$ and such that $\haus^1(\Gamma)=L(\theta) |y-x|$.
\end{lem}
\begin{proof}
There exists a unique point $w$, sometimes called Fermat point, contained in the triangle $xyz$, such that the three angles $\angle ywz$, $\angle zwx$ and $\angle xwy$ are all equal to $2\pi/3$. In particular, $w=z=y$ in the case $\theta=0$, while $w=x$ if $\theta=2\pi/3$. The set $\Gamma$ is then simply the union of the three segments $xw,\, yw$ and $zw$. Calling $L(\theta)$ the length of $\Gamma$ divided by $|y-x|$ (which of course does not depend on $|y-x|$ by rescaling), it is trivial to express $L(\theta)$ as an explicit trigonometric function of $\theta$. The fact that this is a strictly increasing function of $\theta$, with $L(0)=1$ and $L(2\pi/3)=2$, follows then by an elementary calculation.
\end{proof}

We introduce now the (standard) notation of relative perimeter. Given a set $E\subseteq\R^2$ of locally finite perimeter, or a cluster $\E$, and given a Borel set $A\subseteq \R^2$, the \emph{relative perimeter of $E$ (or $\E$) inside $A$} is the measure of the boundary of $E$ (or $\E$) within $A$, i.e.,
\begin{align*}
P(E;A) = \int_{A\cap\partial^* E} h(x)\, d\haus^1(x)\,, &&
P(\E;A) = \int_{A\cap\partial^* \E} h(x)\, d\haus^1(x)\,.
\end{align*}
We conclude this short section by presenting (a very specific case of) a fundamental result due to Vol'pert, see~\cite{Volpert} and also~\cite[Theorem~3.108]{AFP}.
\begin{thm}[Vol'pert]\label{volpert}
Let $E\subseteq\R^2$ be a set of locally finite perimeter, and let $x\in\R^2$ be fixed. Then, for a.e. $r>0$, one has that
\[
\partial^* E \cap \partial B(x,r) = \partial^* \big( E \cap \partial B(x,r)\big)\,.
\]
\end{thm}

Notice that, for almost every $r>0$, both sets in the above equality consist of finitely many points. In particular, $E\cap\partial B(x,r)$ is a subset of the circle $\partial B(x,r)$, and its boundary has to be considered in the $1$-dimensional sense. More precisely, for almost every $r>0$ the set $E\cap \partial B(x,r)$ essentially consists of a finite union of arcs of the circle, and the boundary is simply the union of the endpoints of all of them.

Throughout the rest of the paper, we will often consider intersections of sets with balls. Even if this will not be written every time, we will always consider only balls for which Vol'pert Theorem holds true. So, in particular, in Section~\ref{s:puntitripli}, statements like ``for all radii'' actually mean ``for almost all radii''. This restriction will not cause any drawback in the construction.

\subsection{Finitely many triple points}\label{s:puntitripli}

We now start our construction for proving Theorem~\ref{main}. Through this section and the following one, $\E$ is a fixed, minimal cluster, satisfying the assumptions of Theorem~\ref{main}, and $D$ is a fixed, closed ball. The aim of this section is to show several preliminary properties of $\E$, eventually establishing that $\partial^* \E$ only admits (in $D$) finitely many ``$3$-color points'', see Definition~\ref{def3cp} and Proposition~\ref{dist>R6}. In the sequel we will derive that all the junction points (i.e., the points where at least $3$ of the ${\rm C}^1$ arcs of $\partial\E$ meet) are actually ``$3$-color points'', and in particular triple points (i.e., points where exactly $3$ arcs meet).\par

We set $R_1=\min\{R_\beta,\,R_\eta\}$. In the following, we will define several different values of $R_i$ with $R_1\geq R_2\geq R_3\, \cdots$. Each of these constants will only depend on $\E$, $D$, $g$ and $h$ (the dependence on $\E$, $g$ and $h$ is actually only through the constants $h_{\rm min}$ and $h_{\rm max}$ and on the values of the constants $\Ceeb,\,\Cgr,\,\bar\eps$ and $\eta$ of Definition~\ref{defge}).

\begin{lem}[Small ball competitor]\label{labello}
Let $B(x,r)\subseteq D$ be a ball with $|B(x,r)|<\bar\eps/2$ and $r<R_1$, and let $\E'$ be a cluster which coincides with $\E$ outside $B(x,r)$. There exists another cluster $\E''$ such that $|\E''|=|\E|$, $\E''\cap B(x,r)=\E'\cap B(x,r)$ and, calling $\eps=|\E|-|\E'|$,
\begin{equation}\label{proplab}
P(\E'')\leq P(\E') + \Ceeb |\eps|^\beta \leq P(\E') + \Ceeb (2\Cgr r^\eta)^\beta \,.
\end{equation}
\end{lem}
\begin{proof}
Since
\[
|\eps| \leq \sum_{i=1}^m |\eps_i| \leq 2 |B(x,r)| < \bar\eps\,,
\]
we can apply the $\eps-\eps^\beta$ property to $\E$ with constant $\eps$ and point $x$. Hence, there is another cluster $\F$ such that $\F=\E$ inside $B(x,R_\beta)\supseteq B(x,r)$, and moreover $|\F|=|\E|+\eps$ and $P(\F)\leq P(\E) + \Ceeb |\eps|^\beta$. We define then the cluster $\E''$ as the cluster which coincides with $\E'$ inside $B(x,r)$, and with $\F$ outside of $B(x,r)$. Its volume is
\[\begin{split}
|\E''| &= |\E' \cap B(x,r)| + |\F\setminus B(x,r)|\\
& = |\E\cap B(x,r)| + |\E'| - |\E| + |\E\setminus B(x,r)| + |\F| - |\E|
=|\E|\,.
\end{split}\]
Keeping in mind the growth condition, we have $|\eps| \leq 2 |B(x,r)| \leq 2 \Cgr r^\eta$. As a consequence, the perimeter of $\E''$ can be evaluated as
\[\begin{split}
P(\E'') &= P(\E'; B(x,r)) + P\big(\F;\R^2\setminus B(x,r)\big)\\
&= P(\E; B(x,r)) +P(\E')-P(\E) + P\big(\E;\R^2\setminus B(x,r)\big) + P(\F)-P(\E)\\
&\leq P(\E') + \Ceeb |\eps|^\beta
\leq P(\E') + \Ceeb (2\Cgr r^\eta)^\beta\,.
\end{split}\]
The proof is then concluded.
\end{proof}

\begin{lem}[Length in a ball is controlled by radius]\label{13/2}
There exists a constant $R_2\leq R_1$ such that, for every $B(x,r)\subseteq D$ with $r<R_2$, one has
\begin{equation}\label{step1}
\haus^1(\partial^* \E\cap B(x,r)) < \frac{13}2\, r\,.
\end{equation}
\end{lem}
\begin{proof}
We let $R_2\leq R_1$ be so small that
\begin{align}\label{choiceR2}
\Cgr R_2^\eta <\frac{\bar\eps} 2\,, &&
\omega(2R_2) < \frac{h_{\rm min}}{40}\,, &&
\Ceeb \big( 2 \Cgr R_2^\eta\big)^\beta < h_{\rm min} \frac{R_2}{20}\,.
\end{align}
Notice that the first two inequalities are true for every $R_2$ small enough. The same is true for the third one if $\eta\beta>1$, while if $\eta\beta=1$ the last inequality is true, regardless of $R_2$, since $\Ceeb$ is very small by~(\ref{Cperissmall}).\par

Let now $r<R_2$ and $x\in\R^2$ be as in the claim, and call $\tilde h_{\rm min}=\min\{ h(x),\, x\in \overline{B(x,r)}\}$ and $\tilde h_{\rm max}=\max\{ h(x),\, x\in \overline{B(x,r)}\}$. Let $\E'$ be the cluster defined by $E_1'=E_1\cup B(x,r)$ and $E_i'=E_i\setminus B(x,r)$ for every $2\leq i\leq m$. Clearly
\begin{equation}\label{quellaltro}
P(\E')\leq P(\E) - P(\E; B(x,r)) + 2\pi r \tilde h_{\rm max}\,.
\end{equation}
Let us call $\eps\in\R^m$ the vector given by $\eps_i=|E_i\cap B(x,r)|$ for every $2\leq i\leq m$, and $\eps_1=-|B(x,r)\setminus E_1|$, so that $|\E|=|\E'|+\eps$. Notice that $|B(x,r)|\leq \Cgr r^\eta<\bar\eps/2$ by the first property in~(\ref{choiceR2}). Hence, we can apply Lemma~\ref{labello} to get another cluster $\E''$ satisfying~(\ref{proplab}), so that
\[
P(\E'') \leq P(\E') + \Ceeb (2\Cgr r^\eta)^\beta < P(\E') + h_{\rm min} \,\frac r{20}
\]
by the third property in~(\ref{choiceR2}), which is clearly valid with every $r<R_2$ in place of $R_2$. Putting this estimate together with~(\ref{quellaltro}), and recalling that $P(\E)\leq P(\E'')$ by minimality of $\E$ and since $|\E''|=|\E|$ by Lemma~\ref{labello}, we get
\[
\haus^1(\partial^*\E\cap B(x,r)) \leq \frac{P(\E;B(x,r)) }{\tilde h_{\rm min}}
\leq 2\pi r \,\frac{\tilde h_{\rm max}}{\tilde h_{\rm min}} + \frac{h_{\rm min}}{\tilde h_{\rm min}}\,\frac r{20}
\leq 2\pi r \,\frac{\tilde h_{\rm max}}{\tilde h_{\rm min}} +\frac r{20}
<\frac{13}2\,r\,,
\]
where the last inequality follows from the second property in~(\ref{choiceR2}) since
\[
\tilde h_{\rm max}\leq \tilde h_{\rm min} +\omega(2r) \leq \tilde h_{\rm min} +\frac{ h_{\rm min}}{40}
\leq \frac{41}{40}\,\tilde h_{\rm min}\,.
\]
\end{proof}

\begin{lem}[At most $3$ intersection points]\label{3points}
There exist $R_3\leq R_2$ and $C_2>1$ such that
\begin{equation}\label{6543}
\forall\, r\leq R_3,\, \forall\, B(x,r)\subseteq D,\, \exists\ \frac r{C_2} < \rho < r:\quad \#\Big(\partial^*\E \cap \partial B(x,\rho)\Big)\leq 3\,.
\end{equation}
\end{lem}
\begin{proof}
First of all, we show that~(\ref{6543}) is true with $6$ in place of $3$ by choosing $R_3=R_2$ and $C_2=14$ directly by Lemma~\ref{13/2}. Indeed, suppose that this is false for some $B(x,r)$ as in the claim. Then, for every $r/14<\rho<r$ we have $\#\big(\partial^*\E \cap \partial B(x,\rho)\big)\geq 7$. As a consequence, by coarea formula
\[\begin{split}
\haus^1(\partial^* \E\cap B(x,r))& \geq \haus^1\big(\partial^* \E\cap (B(x,r)\setminus B(x,r/14))\big)
\geq \int_{r/14}^r \haus^0\big(\partial^*\E \cap \partial B(x,\rho)\big) \, d\rho\\
&\geq 7 \, \frac{13}{14}\, r
= \frac{13}2\, r\,,
\end{split}\]
in contradiction with~(\ref{step1}).\par
To conclude the proof is then enough to show that if~(\ref{6543}) holds with some $k\geq 4$ in place of $3$ and with two constants $C_{2,k}$ and $R_{3,k}$, then it also holds with $k-1$ in place of $3$ and with two suitable constants $C_{2,k-1}\geq C_2$ and $R_{3,k-1}\leq R_{3,k}$.\par

Let then $C_{2,k-1}\geq C_{2,k}$ and $R_{3,k-1}\leq R_{3,k}$ be two constants to be specified later, and let $B(x,r)\subseteq D$ with $r\leq R_{3,k-1}$. By assumption, there exists some $r/C_{2,k}<\tilde r<r$ for which $\partial^*\E \cap \partial B(x,\tilde r)$ contains at most $k$ points. If the points are strictly less than $k$ we are done, whatever the choice of $C_{2,k-1}\geq C_{2,k}$ and $R_{3,k-1}\leq R_{3,k}$ is. Assume then that the points are exactly $k$, say $x_1,\, x_2,\, \dots\, ,\, x_k$. We have to find some $r/C_{2,k-1}<\rho<r$ for which
\[
\#\Big(\partial^*\E \cap \partial B(x,\rho)\Big)\leq k-1\,.
\]
Assume then by contradiction that for every such $\rho$ (so in particular for every $r/C_{2,k-1}<\rho<\tilde r$) the opposite inequality holds, then by coarea formula again we deduce
\begin{equation}\label{perintilder}\begin{split}
\haus^1(\partial^* \E \cap B(x,\tilde r)) &\geq \haus^1\Big(\partial^* \E\cap \big(B(x,\tilde r)\setminus B(x,r/C_{2,k-1})\big)\Big)
\geq k \bigg(\tilde r- \frac r{C_{2,k-1}}\bigg)\\
&\geq k\tilde r \bigg(1-\frac{{C_{2,k}}}{C_{2,k-1}}\bigg)\,.
\end{split}\end{equation}
Up to renumbering, we can assume that the two points among $x_1,\, x_2,\, \dots\, x_k$ at minimal distance are $x_1$ and $x_2$, so in particular
\[
\angle{x_1}x{x_2} \leq \frac{2\pi}k<\frac {2\pi}3 \,.
\]
Let us define the set $\Sigma\subseteq B(x,\tilde r)$ as the union of the $k-2$ segments $x_i x$ for $i\geq 3$ together with the set $\Gamma$ given by Lemma~\ref{steiner} by setting $y=x_1$ and $z=x_2$. Recall that the set $\Gamma$ is contained in the triangle $xx_1x_2$ by Lemma~\ref{steiner}, hence it does not intersect the segments $x_ix$ with $i\geq 3$. As a consequence, the union of $\Sigma$ with $\partial^*\E\setminus B(x,\tilde r)$ is the boundary of a uniquely defined cluster $\E'$ which coincides with $\E$ outside of $B(x,\tilde r)$. Notice that by construction, Lemma~\ref{steiner} and~(\ref{perintilder}) one has
\[\begin{split}
P(\E')-P(\E) &\leq \tilde h_{\rm max} \haus^1(\Sigma) - \tilde h_{\rm min} \haus^1(\partial^*\E\cap B(x,\tilde r))\\
& \leq\tilde h_{\rm max} \tilde r \bigg[ k-2+L\bigg(\frac{2\pi}k\bigg) -
\frac{\tilde h_{\rm min}}{\tilde h_{\rm max}}\, k\bigg(1-\frac{{C_{2,k}}}{C_{2,k-1}}\bigg)\bigg]\,,
\end{split}\]
having set again $\tilde h_{\rm min}=\min\{ h(x),\, x\in \overline{B(x,r)}\}$ and $\tilde h_{\rm max}=\max\{ h(x),\, x\in \overline{B(x,r)}\}$. Since $\tilde r<r\leq R_{3,k-1}$ and
\[
\frac{\tilde h_{\rm min}}{\tilde h_{\rm max}} \geq 1- \frac{\omega(2r)}{h_{\rm min}}\geq 1- \frac{\omega(2R_{3,k-1})}{h_{\rm min}}\,,
\]
then, up to choose $R_{3,k-1}$ small enough and $C_{2,k-1}$ big enough, this yields
\begin{equation}\label{stimac}
P(\E')-P(\E) \leq -c\tilde r \tilde h_{\rm max}\leq - c \tilde r h_{\rm min}
\end{equation}
with $2c=2-L(2\pi/k)$. We apply again Lemma~\ref{labello} with $\eps=|\E|-|\E'|$ to get a cluster $\E''$ with $|\E''|=|\E|$, $B(x,\tilde r)\cap\E''=B(x,\tilde r)\cap\E'$ and such that
\[
P(\E'')- P(\E') \leq \Ceeb (2\Cgr r^\eta)^\beta\,.
\]
By the minimality of $\E$, putting this inequality together with~(\ref{stimac}) we obtain
\[
\Ceeb (2 \Cgr r^\eta)^\beta \geq c\tilde r h_{\rm min}
\geq \frac{cr h_{\rm min}}{C_{2,k}}\,,
\]
hence
\[
r^{\eta \beta-1} \geq \frac{c \, h_{\rm min}}{C_{2,k}\Ceeb 2^\beta \Cgr^\beta}\,.
\]
If $\eta\beta>1$, this gives a lower bound to $r$, hence we have the searched contradiction, up to possibly decrease $R_{3,k-1}$. Instead, if $\eta\beta=1$, the searched contradiction follows since $\Ceeb$ is very small by~(\ref{Cperissmall}).
\end{proof}

\begin{lem}[No-islands]\label{noisland}
There exists $R_4\leq R_3$ such that for every $x\in D$ and almost every $r\leq R_4$ so that $B(x,r)\subseteq D$, if for some $0\leq i\leq m$ one has
\begin{equation}\label{intint}
|E_i \cap B(x,r)|>0\,,
\end{equation}
then also
\begin{equation}\label{intbor}
\haus^1(E_i \cap \partial B(x,r))>0\,.
\end{equation}
\end{lem}
Notice that, in this lemma, $i$ can also attain the value $0$. Recall that $E_0$ has been defined as $\R^2\setminus (\cup_{i=1}^m E_i)$. We remark that, in~\cite{Leonardi}, the author shows a ``no-island'' property in a clustering problem, calling it ``absence of infiltration'', and yielding almost everywhere regularity of the boundary.
\begin{proof}[Proof of Lemma~\ref{noisland}]
Let $x$ be a point in $D$, and let $B(x,r)\subseteq D$ be a ball with $r\leq R_3$ for which Vol'pert Theorem~\ref{volpert} holds true with each of the sets $E_i$'s, $i=1,\dots,m$ (this only rules out a negligible quantity of radii). Let us assume that, for some $0\leq i \leq m$, (\ref{intint}) holds while~(\ref{intbor}) does not. We have to prove that this is contradictory, provided that $r$ is is smaller than some $R_4\leq R_3$ that we are going to specify later.\par
Let us call $F=E_i\cap B(x,r)$. Since~(\ref{intbor}) is false and $r$ satisfies Vol'pert Theorem then, up to $\haus^1$-negligible subsets, $\partial^* F\subseteq B(x,r)$, and in particular
\[
\partial^* F \subseteq \bigcup_{\doppio{j\in \{0,\,1,\,\dots\,,\,m\}}{j\neq i}}\ \partial^* E_j\,,
\]
so that for some $0\leq \ell\leq m$, $\ell\neq i$ we have
\[
\haus^1(\partial^* F \cap \partial^* E_\ell)\geq \frac 1m\, \haus^1(\partial^* F)\geq \frac 1{h_{\rm max}m}\, P(F)\,.
\]
Let us then define the cluster $\E'$ as the cluster such that $E_i'=E_i\setminus F$, $E_j'=E_j$ for every $j\notin \{i,\,\ell\}$ and $E_\ell'=E_\ell\cup F$. By construction and by Lemma~\ref{disgusting} we have
\begin{equation}\label{nomesensato}\begin{split}
P(\E') &\leq P(\E) - h_{\rm min}\, \haus^1(\partial^* F\cap \partial^* E_\ell) \leq P(\E) - \frac {h_{\rm min}}{h_{\rm max}m}\, P(F)\\
&\leq P(\E) - \frac {h_{\rm min}^2}{h_{\rm max}\Cgr^{1/\eta} m}\, |F|^{1/\eta}\,.
\end{split}\end{equation}
Let us define $\eps=|\E|-|\E'|$, so that $|\eps|\leq 2|F|$, and the latter is strictly positive by~(\ref{intint}). Applying again Lemma~\ref{labello}, we get a cluster $\E''$ with $|\E''|=|\E|$, $B(x,\tilde r)\cap\E''=B(x,\tilde r)\cap\E'$, and
\[
P(\E'')\leq P(\E') + \Ceeb |\eps|^\beta \leq P(\E') + \Ceeb 2^\beta |F|^\beta\,.
\]
Putting this inequality together with~(\ref{nomesensato}), by the optimality of $\E$ we find
\[
(\Cgr r^\eta)^{1/\eta-\beta} \leq |B(x,r)|^{1/\eta-\beta} \leq |F|^{1/\eta-\beta} \leq \frac{h_{\rm max}}{h_{\rm min}^2}\,\Cgr^{1/\eta}\Ceeb 2^\beta m\,.
\]
As usual, we have to distinguish two cases in order to conclude. If $\eta>1/\beta$, then the inequality implies that $r$ cannot be too small, hence $r>R_4$ for some $R_4\leq R_3$. If $\eta=1/\beta$, then the inequality is false because $\Ceeb$ is very small by~(\ref{Cperissmall}).
\end{proof}

\begin{rmk}\label{noislandG}
Notice that the claim of the above lemma can be trivially generalized as follows. Let $E\subseteq\R^2$ be a given set, let $G\subseteq B(x,r)$ be a set with Lipschitz boundary, and assume that the analogous of the Vol'pert Theorem holds, that is, $\partial^*E_i\cap \partial G=\partial^*(E_i\cap \partial G)$ for every $1\leq i \leq m$. If $|E_i\cap G|>0$, then also $\haus^1(E_i\cap \partial G)>0$. The proof remains exactly the same, one only has to substitute $B(x,r)$ with $G$ everywhere.
\end{rmk}

\begin{cor}[At most $3$ colors]\label{atmost3}
For every $B(x,R_4)\subseteq D$ one has
\[
\# \Big\{0\leq i\leq m:\, |E_i \cap B(x,R_4/C_2)|>0 \Big\}\leq 3\,.
\]
\end{cor}
\begin{proof}
We apply Lemma~\ref{3points} to the ball $B(x,R_4)$, finding some $R_4/C_2<\rho<R_4$ for which $\partial^* \E\cap \partial B(x,\rho)$ consists of at most three points. As a consequence, there are at most three different indices $0\leq i \leq m$ such that $E_i\cap \partial B(x,\rho)$ has positive $\haus^1$-measure (keep in mind that Vol'pert Theorem~\ref{volpert} holds true for the ball $B(x,\rho)$). Since $\rho<R_4$, by Lemma~\ref{noisland} we obtain that $|E_i\cap B(x,\rho)|$ can be strictly positive for at most three different indices $0\leq i\leq m$, and since $\rho>R_4/C_2$ the proof is concluded.
\end{proof}

\begin{defn}[$3$-color point]\label{def3cp}
A point $x\in \R^2$ is said a \emph{$3$-color point} if, for every $r>0$, we have
\begin{equation}\label{quellali}
\#\big\{0\leq i\leq m:\, |E_i\cap B(x,r)|>0\big\}\geq 3\,.
\end{equation}
\end{defn}

Notice that, in view of Corollary~\ref{atmost3}, for every $3$-color point $x\in D$ the sets $E_i$ satisfying~(\ref{quellali}) are actually exactly $3$ for every $r<\min\{R_4/C_2,\, {\rm dist}(x,\partial D)\}$. This also motivates the name.\par

We can now improve the result of Lemma~\ref{3points} for balls centered at a $3$-color point, namely, in~(\ref{6543}) the possibly large constant $C_2$ can be replaced by any constant strictly larger than $1$. We are going to prove the result with constant $25/24$ just because this is the value that we will need later, but it is clear from the proof that nothing changes with any other constant strictly larger than $1$.
\begin{lem}\label{9/10}
There exists $R_5\leq R_4$ such that, for any $3$-color point $x\in D$ and any $r\leq \min\{R_5,\, {\rm dist}(x,\partial D)/C_2\}$, there is $24r/25<\bar\rho<r$ such that $\#\Big(\partial^*\E \cap \partial B(x,\bar\rho)\Big)= 3$.
\end{lem}
\begin{proof}
Let $R_5\leq R_4$ be a constant to be specified later, and let $x$ and $r$ be as in the claim. By Lemma~\ref{noisland}, we know that
\begin{align}\label{alm3}
\#\Big(\partial^*\E \cap \partial B(x,s)\Big)\geq 3 && \forall\ 0<s\leq \min\{R_4,\, {\rm dist}(x,\partial D)\}\,.
\end{align}
By Lemma~\ref{3points}, we find some $r<\rho< C_2 r$ such that
\[
\#\Big(\partial^*\E \cap \partial B(x,\rho)\Big)\leq 3\,,
\]
and the number is in fact exactly $3$ by~(\ref{alm3}). Notice that we have applied Lemma~\ref{3points} with $C_2 r$ in place of $r$, and this is possible only if $C_2 r\leq \min\{R_3,\, {\rm dist}(x,\partial D)\}$, which in turn is admissible up to choose $R_5\leq R_3/C_2$.\par
Let us now call
\[
\mu = \haus^1 \Big(\Big\{ 0<s<\rho:\, \#\Big(\partial^*\E \cap \partial B(x,s)\Big)\geq 4 \Big\}\Big)\,,
\]
so that the thesis follows as soon as we show that, with the right choice of $R_5$, $\mu<r/25$. In view of~(\ref{alm3}), by coarea formula we can estimate
\[
P(\E;B(x,\rho))\geq \tilde h_{\rm min} \big(3(\rho-\mu)+4\mu\big) =\tilde h_{\rm min} (3\rho+\mu)\,,
\]
having again defined $\tilde h_{\rm min}=\min\{ h(x),\, x\in \overline{B(x,\rho)}\}$ and $\tilde h_{\rm max}=\max\{ h(x),\, x\in \overline{B(x,\rho)}\}$. Let us now define $\E'$ as the cluster which coincides with $\E$ outside of $B(x,\rho)$ and such that $\partial^* \E'\cap B(x,\rho)$ is done by the three segments joining the three points of $\partial^*\E\cap \partial B(x,\rho)$ with $x$. Observe that
\[
P(\E';B(x,\rho))\leq 3\tilde h_{\rm max} \rho \leq 3\tilde h_{\rm min} \rho + 3\omega(2\rho)\rho\,.
\]
As usual, we apply Lemma~\ref{labello} to get a competitor $\E''$ with $|\E''|=|\E|$ and
\[
P(\E'') \leq P(\E') + \Ceeb (2 \Cgr \rho^\eta)^\beta\,.
\]
Putting together the last three estimates, since $P(\E)\leq P(\E'')$ and $h_{\rm min}\leq \tilde h_{\rm min}$ we find
\[
\mu \leq \frac \rho{h_{\rm min}}\, \Big( 3\omega(2\rho)+ 2^\beta \Ceeb \Cgr^\beta \rho^{\eta\beta-1}\Big)
\leq \frac {C_2 r}{h_{\rm min}}\, \Big( 3\omega(2\rho)+ 2^\beta \Ceeb \Cgr^\beta \rho^{\eta\beta-1}\Big)\,,
\]
hence the thesis reduces to check that the term in parentheses in the above estimate can be taken smaller than $h_{\rm min}/(25 C_2)$. Concerning $3\omega(2\rho)\leq 3\omega(2C_2 R_5)$, this is arbitrarily small as soon as $R_5$ is small enough, so we can suppose that this is smaller $h_{\rm min}/(50C_2)$, and to conclude we need the second term in parenthesis to be smaller than $h_{\rm min}/(50C_2)$. This is clearly true for $R_5$ small enough if $\eta\beta>1$, while in the case $\eta\beta=1$ this is true regardless of $R_5$ thanks to~(\ref{Cperissmall}).
\end{proof}

\begin{prop}[$3$-color points are a positive distance apart]\label{dist>R6}
There exists $R_6\leq R_5$ such that any two $3$-color points $x,\,x' \in D$ with $B(x,R_5)\subseteq D$ have distance at least $R_6$.
\end{prop}
\begin{proof}
Let us assume by contradiction the existence of two $3$-color points $x$ and $x'$ as in the claim with $d:=|x-x'|<R_6$, where $R_6\leq R_5$ is a constant to be specified later. Since the proof is a bit involved, we divide it in few steps for the sake of clarity.
\step{I}{A circle with three boundary points at around $120^\circ$.}
We start by applying Lemma~\ref{9/10} to the point $x$ and with $r=\frac 54 d$, which is admissible as soon as $R_6\leq \frac 45\, R_5$, finding a radius $\rho$ with
\begin{equation}\label{850}
\frac 65\,d = \frac {24}{25} \cdot \frac 54\, d< \rho < \frac 54\, d
\end{equation}
such that $\partial^* \E \cap \partial B(x,\rho)$ contains exactly three points, say $x_1,\,x_2,\,x_3$. We want to show that these three points are close to be vertices of an equilateral triangle. More precisely, we will prove that
\begin{equation}\label{stima119}
\theta=\min \Big\{\angle {x_1} x {x_2},\, \angle {x_2}x{x_3},\, \angle {x_3}x{x_1}\Big\} > \bigg(\frac 23 -\frac 1{15}\bigg)\,\pi=\frac 35\,\pi\,.
\end{equation}
Indeed, assuming just to fix the ideas that $\theta=\angle {x_1}x{x_2}$, we define the cluster $\E'$ which coincides with $\E$ outside of $B(x,\rho)$ and such that $\partial^* \E' \cap B(x,\rho)$ is done by the segment $xx_3$ and the set $\Gamma$ given by Lemma~\ref{steiner} with $y=x_1$ and $z=x_2$. By Lemma~\ref{labello} we have a cluster $\E''$ with $|\E''|=|\E|$ such that
\begin{equation}\label{stima120}\begin{split}
P(\E)&\leq P(\E'') \leq P(\E') + \Ceeb (2 \Cgr \rho^\eta)^\beta\\
&\leq P(\E) + \rho \bigg(\tilde h_{\rm max} \big(1 + L(\theta)\big) - 3\tilde h_{\rm min}+ 2^\beta\Ceeb \Cgr^\beta \rho^{\eta\beta-1}\bigg)\,,
\end{split}\end{equation}
where we have again defined $\tilde h_{\rm min}=\min\{ h(x),\, x\in \overline{B(x,\rho)}\}$ and $\tilde h_{\rm max}=\max\{ h(x),\, x\in \overline{B(x,\rho)}\}$, and where we have used the fact that
\[
P(\E;B(x,\rho)) \geq 3\tilde h_{\rm min} \rho\,,
\]
which follows from Lemma~\ref{noisland} because $x$ is a $3$-color point and by coarea formula. Arguing as usual, a straightforward computation from~(\ref{stima120}) gives
\[
2- L(\theta) \leq \frac 1{h_{\rm min}}\, \Big( 3\omega(2\rho) + 2^\beta\Ceeb \Cgr^\beta \rho^{\eta\beta-1} \Big)\,.
\]
As already done several times, both if $\eta\beta>1$ and if $\eta\beta=1$ we get that, provided $R_6$ is small enough, $L(\theta)$ is as close as $2$ as we wish, in particular we can assume that~(\ref{stima119}) holds true.\par

For later use, we remark that from~(\ref{stima120}) and by definition of $\E'$ we have
\begin{equation}\label{fromabove}\begin{split}
P(\E;B(x,\rho)) &\leq P(\E';B(x,\rho)) + \Ceeb (2 \Cgr \rho^\eta)^\beta
\leq 3 \tilde h_{\rm max}\rho + \Ceeb (2 \Cgr \rho^\eta)^\beta\\
&< \frac{19}6\, \tilde h_{\rm min} \rho\,,
\end{split}\end{equation}
again up to possibly decrease the value of $R_6$.

\step{II}{The estimate~(\ref{estiQaQb}).}
In this step we define the ``curved rectangles'' $Q_a$ and $Q_b$ and we prove the estimate~(\ref{estiQaQb}). The situation is depicted in Figure~\ref{FigPosDist}.\par
\begin{figure}[thbp]
\begin{tikzpicture}[>=>>>]
\filldraw[fill=blue!25!white, draw=black, line width=.7pt] (0.78,-4.43) arc (280:325:4.5) -- (2.46,-1.72) arc(325:280:3) -- cycle;
\filldraw[fill=blue!25!white, draw=black, line width=.7pt] (-3.69,2.58) arc (145:190:4.5) -- (-2.95,-.52) arc(190:145:3) -- cycle;
\draw[line width=1] (0,0) circle (3);
\draw[line width=1] (0,0) circle (4.5);
\fill (0,0) circle (2pt);
\draw (0,0) node[anchor=north east] {$x$};
\draw[<->] (0,0) -- (0,3);
\draw (0,1.5) node[anchor=east] {$2d-\rho$};
\draw[<->] (0,0) -- (3.75,0);
\draw (1.7,0) node[anchor=south] {$d$};
\fill (3.75,0) circle (2pt);
\draw (3.75,0) node[anchor=north] {$x'$};
\draw (3.75,0) circle (0.75);
\draw[<->] (5,0.75) -- (5,-0.75);
\draw (5,0) node[anchor=west] {$2\rho-2d$};
\draw[<->] (0,0) -- (3.18,3.18);
\draw (1.59,1.59) node[anchor=south] {$\rho$};
\fill (-4.39,.97) circle (2pt);
\draw (-4.39,.97) node[anchor=south east] {$b$};
\draw (-4,2.2) node[anchor=east] {$Q_b$};
\draw (3.6,-3.7) node[anchor=east] {$Q_a^+$};
\draw (2.4,-4.4) node[anchor=east] {$Q_a^-$};
\draw[dashed] (-2.95,-.52) -- (0,0) -- (-4.39,.97);
\draw[dashed] (.78,-4.43) -- (0,0) -- (2.42,-3.8);
\fill (2.42,-3.8) circle (2pt);
\draw (2.38,-3.8) node[anchor=north west] {$a$};
\draw[scale=1.2] (-.98,.22) arc(167.5:190:1);
\draw (-1.2,0) node[anchor=east] {$\phi$};
\draw[scale=.266] (0.78,-4.43) arc(280:302.5:4.5);
\draw (.85,-1.5) node[anchor=east] {$\phi$};
\draw[scale=.62,<->] (0.78,-4.43) arc(280:302.5:4.5);
\draw (-1.4,-2.1) node[anchor=west] {$2\rho-2d$};
\draw[->] (.2,-2.1) arc(90:30:.9);
\draw (0.59,-3.35) arc(280:325:3.4);
\draw[<->] (0,0) -- (2.6,-2.19);
\draw (1.2,-1.15) node[anchor=south west] {$\bar s$};
\draw[line width=1] (2.4,-2.4) -- (3.18,-3.18);
\draw[line width=1] (.88,-3.28) -- (1.16,-4.35);
\draw (3.1,-1.8) node[anchor=south west] {$S_{\theta^+}$};
\draw[->] (3.3,-1.8) -- (2.84,-2.74);
\draw (0,-4) node[anchor=east] {$S_{\theta^-}$};
\draw[->] (-.15,-3.9) -- (.94,-3.84);
\end{tikzpicture}
\caption{The situation in Proposition~\ref{dist>R6}. Here $\phi$ is chosen so that the rectangular sectors at $a$ and $b$ corresponding to $\phi$ are disjoint and do not intersect the ball $B(x',\rho-d)$ at the tricolor point $x'$.}\label{FigPosDist}
\end{figure}
We start by observing that, since both $x$ and $x'$ are $3$-color points, by Lemma~\ref{noisland},
\begin{align}\label{esti2balls}
P\big(\E; B(x',\rho -d )\big) \geq 3 \tilde h_{\rm min} (\rho-d)\,, &&
P\big(\E; B(x,\rho -2(\rho-d) )\big) \geq 3 \tilde h_{\rm min} (2d-\rho)\,.
\end{align}
Among $x_1,\,x_2$ and $x_3$, let $a$ and $b$ be the two points having maximal distance from $x'$, let $\phi$ be the angle defined by the property
\begin{equation}\label{defphi}
(2d-\rho) \phi = 2\rho - 2d\,,
\end{equation}
and let $Q_a$ and $Q_b$ be the curved rectangles defined by
\begin{align*}
Q_a &= \Big\{ z \in B(x,\rho)\setminus B(x,2d-\rho):\, |\angle zxa|<\phi \Big\}\,, \\
Q_b &= \Big\{ z \in B(x,\rho)\setminus B(x,2d-\rho):\, |\angle zxb|<\phi \Big\}\,.
\end{align*}
Notice that the angles $\angle ax{x'}$ and $\angle {x'}xb$ are both at least $3\pi/10$ by~(\ref{stima119}), while $\phi<2/3\approx 0.21\pi$ since $d>4\rho/5$. As a consequence, a simple trigonometric computation ensures that, as in the figure, $Q_a$ and $Q_b$ are disjoint and they do not intersect the ball $B(x',\rho-d)$. The goal of this step is to show that
\begin{align}\label{estiQaQb}
P(\E; Q_a) \geq \tilde h_{\rm min}(2\rho-2d)\,, && P(\E; Q_b) \geq \tilde h_{\rm min}(2\rho-2d)\,.
\end{align}
Let us prove the estimate for $Q_a$, since there is no difference with the case of $Q_b$. If for almost every $2d-\rho<s<\rho$ one has $\#(\partial^*\E\cap \partial B(x,s)\cap Q_a)\geq 1$, the claim directly follows by integration.\par

Suppose then the existence of some $2d-\rho<\bar s<\rho$ such that the arc $\partial B(x,\bar s)\cap Q_a$ does not intersect $\partial^* \E$. Let us subdivide $Q_a=Q_a^+\cup Q_a^-$, where $Q_a^\pm$ are the two parts in which $Q_a$ is divided by the segment $\{z: z-x=\sigma(a-x),\, \frac{2d}\rho-1<\sigma<1\}$. For every $0<\theta<\phi$, let us now call $S_\theta$ the closed segment made by all the points $z\in Q_a^+$ such that $\angle axz=\theta$ and $\bar s\leq |z-x|\leq \rho$, and similarly for every $-\phi<\theta<0$ we call $S_\theta$ the closed segment made by all the points $z\in Q_a^-$ such that $\angle zxa=-\theta$ and $\bar s<|z-x|<\rho$. If $\#(\partial^*\E\cap S_\theta)\geq 1$ for almost every $0<\theta<\phi$, or for almost every $-\phi<\theta<0$, then again by integration and recalling~(\ref{defphi}) we obtain the searched estimate for $P(\E;Q_a)$.\par

We are then only left to consider the case when two angles $\theta^\pm$ with $-\phi<\theta^-<0<\theta^+<\phi$ exist such that both the segments $S_{\theta^+}$ and $S_{\theta^-}$ have no intersection with $\partial^*\E$. In this case, putting together the two arcs $\partial B(x,\rho)\cap \bigcup_{\theta^-\leq \theta\leq \theta^+} S_\theta$ and $\partial B(x,\bar s)\cap \bigcup_{\theta^-\leq \theta\leq \theta^+} S_\theta$, and the two segments $S_{\theta^-}$ and $S_{\theta^+}$, we obtain a Lipschitz, closed loop which intersects $\partial^*\E$ in a single point, namely, $a$. However, we can easily notice that this is impossible, because for any such loop the number of intersections with $\partial^*\E$ must be either empty, or done by at least two points. Let us be more precise: if we call $G$ the Lipschitz set whose boundary is the above-mentioned loop, we can assume without loss of generality that the Vol'pert property holds for $G$ and each of the sets $E_i,\, 1\leq i \leq m$ (this only rules out a negligible set of directions and radii). In other words, for every $1\leq i \leq m$ we have that $\partial^* E_i \cap \partial G= \partial^*(E_i\cap \partial G)$. However, $E_i\cap \partial G$ is a subset of the Lipschitz curve $\partial G$ .Therefore, its boundary is empty if this subset is either empty or coincides with the whole curve $\partial G$; and otherwise, it is done by at least two points. Since this is true for any $1\leq i \leq m$, we also get that $\partial^* \E \cap \partial G$ must be either empty, or done by at least two points. The contradiction concludes the proof of the validity of~(\ref{estiQaQb}).

\step{III}{Estimate on $P(\E;B(x,\rho))$ from below.}
In this last quick step we find an estimate of $P(\E;B(x,\rho))$ from below, which gives a contradiction with the estimate from above found in Step~I, concluding the proof. Since the curved rectangles $Q_a$ and $Q_b$ and the balls $B(x,2d-\rho)$ and $B(x',\rho-d)$ are pairwise disjoint, by~(\ref{esti2balls}), (\ref{estiQaQb}) and~(\ref{850}) we obtain
\[
P(\E;B(x,\rho)) \geq \tilde h_{\rm min} (4\rho-d)
\geq \frac{19}6 \, \tilde h_{\rm min} \rho\,,
\]
against~(\ref{fromabove}).
\end{proof}

\subsection{Interface regularity}\label{s:interface}

This section is devoted to prove the following regularity result for the boundary of the optimal cluster $\E$.

\begin{prop}[${\rm C}^{1,\gamma}$ regularity]\label{step6}
There exists an increasing function $\xi:\R^+\to\R^+$ with $\lim_{r\to 0^+} \xi(r)=0$ such that the following property holds. Let $B(\bar x,\bar r)\subseteq D$ be a ball such that
\begin{align}\label{only2colors}
\# \Big\{0\leq i\leq m:\, |E_i\cap B(\bar x,\bar r)|>0\Big\} \leq 2\,, &&
\#\, \big(\partial^* \E\cap \partial B(\bar x,\bar r)\big) <+\infty\,.
\end{align}
Then, $\partial^*\E\cap B(\bar x,\bar r)$ is a finite union of ${\rm C}^1$ pairwise disjoint relatively closed curves such that, calling $\tau(x)\in\P^1$ the direction of the tangent vector at any $x\in\partial^*\E\cap B(\bar x,\bar r)$, one has
\begin{equation}\label{uniformC1}
|\tau(y)-\tau(x)| \leq \xi(|y-x|)
\end{equation}
for every $x,\,y\in \partial^*\E\cap B(\bar x,\bar r)$. Moreover, if $\eta\beta>1$ and $h$ is locally $\alpha$-H\"older continuous, then it is possible to take $\xi(t)=K t^\gamma$ with some $K>0$ and
\begin{equation}\label{heregamma}
\gamma=\frac 12\, \min\{\eta\beta-1,\, \alpha\}\,,
\end{equation}
so that in particular $\partial^*\E\cap B(\bar x,\bar r)$ is ${\rm C}^{1,\gamma}$.
\end{prop}

\begin{lem}[Almost alignment on a circle with two boundary points]\label{aa2bp}
There exist $R_7<R_6$ and a function $\xi_1:\R^+\to\R^+$ as in Proposition~\ref{step6} and satisfying the Dini property such that the following holds. Let $B(\bar x,\bar r)$ be as in Proposition~\ref{step6}, and let $x\in\partial^*\E$ and $r<R_7$ be such that
\begin{align*}
B(x,r)\subseteq B(\bar x,\bar r)\,, &&
\partial B(x,r) \cap \partial^* \E = \{a,\,b\}
\end{align*}
where $a,\,b$ are two distinct points in $\partial^*\E$, such that $0<\angle axb\leq \pi$. Then, calling $d=\haus^1(\partial^*\E\cap B(x,r))-|b-a|$, we have
\begin{align}\label{xi1}
|\angle axb - \pi | \leq \xi_1(r)\,, &&
d \leq \frac r 6 \, \xi_1(r)^2\,.
\end{align}
\end{lem}
\begin{proof}
We define $\E'$ as the cluster which coincides with $\E$ outside of $B(x,r)$, and such that $\partial \E'\cap B(x,r)$ is done by the segment $ab$. We let $d=\haus^1(\partial^*\E\cap B(x,r))-|b-a|$. Observe that, since $x\in \partial^*\E$, the fact that $r<R_6$ together with Lemma~\ref{noisland} ensures that for almost each $0<s<r$ the set $\partial^*\E\cap \partial B(x,s)$ is non-empty, thus by Vol'pert Theorem it contains at least $2$ points. As a consequence, by coarea formula, $\haus^1(\partial\E'\cap B(x,r))\geq 2r$, hence $d\geq 2r-|b-a|$. Setting $\tilde h_{\rm min}$ and $\tilde h_{\rm max}$ as usual, we have
\[
P(\E')-P(\E) \leq \tilde h_{\rm max} |b-a| - \tilde h_{\rm min} (|b-a|+d)\,.
\]
As a consequence, minding that $\E\Delta \E' \subseteq B(x,r)$ and $|B(x,r)|\leq \Cgr r^\eta$, by Lemma~\ref{labello} --notice that in~(\ref{proplab}) one can clearly use $\Ceeb[|\eps|]$ in place of $\Ceeb$-- we readily obtain
\[
d\leq \frac 1{h_{\rm min}}\, \Big(2 r \omega(2r) + \Ceeb[2\Cgr r^\eta] (2\Cgr r^\eta)^\beta\Big)\,.
\]
Let us now set
\[
\xi_1(r) = \bigg(\frac 6{h_{\rm min}}\, \Big(2\omega(2r) + 2^\beta \Ceeb[2\Cgr r^\eta] \Cgr^\beta r^{\eta\beta-1}\Big)\bigg)^{1/2}\,,
\]
so that the right estimate in~(\ref{xi1}) holds true. The left one then easily follows since
\[
\frac r6\, \xi_1(r)^2 \geq d \geq 2r - |b-a| = 2r \Big(1 - \sin\big(\angle axb/2\big) \Big) 
\geq \frac r6\, (\angle axb - \pi)^2\,.
\]
Hence, to conclude we only have to check the properties of $\xi_1$. The fact that $\xi_1$ is increasing is true by construction, and the fact that, as $r\searrow 0$, it goes to $0$ is true since $\omega(r)\searrow 0$, and $r^{\eta\beta-1}\searrow 0$ if $\eta\beta>1$, while $\Ceeb[2\Cgr r^\eta]\searrow 0$ if $\eta\beta=1$ by assumption. In addition, since $r^\eta\leq r$ because $\eta\geq 1$, then
\[
\xi_1(r) \lesssim \sqrt{\omega(2r) + \Ceeb[2\Cgr r^\eta] r^{\eta\beta-1}}
\lesssim \sqrt{\omega(2r)} + \sqrt{\Ceeb[2\Cgr r] r^{\eta\beta-1}}\,.
\]
The Dini property of $\xi_1$ then readily follows. Indeed, the Dini property of $\sqrt{\omega(2r)}$ is true since by assumption $h$ is locally $1/2$-Dini continuous. Moreover, the Dini property of $\sqrt{\Ceeb[2\Cgr r] r^{\eta\beta-1}}$ is clear if $\eta\beta>1$, while it comes from the $1/2$-Dini property of $t\mapsto\Ceeb[t]$ if $\eta\beta=1$. Finally, if $\eta\beta>1$ and $h$ is locally $\alpha$-H\"older then we have
\[
\xi_1(r) \lesssim \sqrt{\omega(2r) + r^{\eta\beta-1}} \lesssim \sqrt{r^\alpha + r^{\eta\beta-1}}\approx r^\gamma\,,
\]
with $\gamma$ given by~(\ref{heregamma}).
\end{proof}

\begin{lem}[Almost alignment on every circle]\label{xi2}
There exists a function $\xi_2:\R^+\to\R^+$ as in Lemma~\ref{aa2bp} such that the following holds. Let $B(\bar x,\bar r)$ be as in Proposition~\ref{step6}, and let $x\in\partial^*\E$ and $r<R_7/(2C_2)$ be such that $B(x,2C_2 r)\subseteq B(\bar x,\bar r)$. Then, there exists a direction $\tau(x,r)\in\P^1$ such that every $y\in\partial^*\E\cap\partial B(x,r)$ satisfies
\begin{equation}\label{alignmenteveryc}
|\zeta(y-x)- \tau(x,r)| \leq \xi_2(r)\,,
\end{equation}
where $\zeta(v)=\big[ \, v/|v|\,\big] \in\P^1$ is the direction of any vector $v\in\R^2\setminus \{0\}$. 
Moreover, for every $r' \in [r/2,r)$
\begin{equation}\label{taucont}
|\tau(x,r)-\tau(x,r')| \leq 2 C_2\xi_1(2C_2 r)+2\xi_2(r)\,.
\end{equation}
\end{lem}

\begin{proof}
Since $r<R_7/(2C_2)<R_3/(2C_2)$, by Lemma~\ref{3points} there exists some $2r< \rho < 2C_2 r$ such that $\partial B(x,\rho)\cap\partial^*\E$ contains at most three points. We claim that these points are actually $2$. In fact, since $x\in \partial^* \E$ and $\rho<R_4$ then there must be at least two such points by Lemma~\ref{noisland}. On the other hand, in view of~(\ref{only2colors}) and again by Lemma~\ref{noisland}, we obtain that $\#\big\{0\leq i \leq m:\, \haus^1(E_i\cap \partial B(x,\rho))>0\big\}\leq 2$, so that the number of points of $\partial B(x,\rho)\cap\partial^*\E$ must be even (keep in mind that, as always, Vol'pert Theorem holds for $B(x,\rho)$). The claim is then proved, and we can then call $a$ and $b$ these two points, and define $\tau(x,r)\in\P^1$ the direction of the segment $ab$. Notice that the vector $\tau(x,r)$ depends on $x$, on $r$, and on the choice of $2r<\rho<2C_2 r$. The vector $\tau(x)$ of Proposition~\ref{step6}, instead, will only depend on $x$, as one can clearly deduce from~(\ref{uniformC1}).\par

Let now $y\in\partial^* \E\cap \partial B(x,r)$ be given, let us call $d_1=|y-a|$ and $d_2=|y-b|$ and assume, without loss of generality, that $d_1\leq d_2$. For every $0<s<d_2$, we call $G_s=B(y,s)\cap B(x,\rho)$. Since $y\in\partial^* \E$, by Lemma~\ref{noisland} and keeping in mind also Remark~\ref{noislandG}, we obtain that $\Gamma_s:=\partial G_s\cap\partial^* \E$ contains at least two points. Since $s<d_2$, $\Gamma_s$ cannot contain $b$, and it cannot contain $a$ if $s<d_1$. Recalling that $\partial B(x,\rho)\cap\partial^* \E=\{a,\,b\}$, we deduce that $\Gamma_s\cap B(x,\rho)$ contains at least two points for $0<s<d_1$, and at least one point for $d_1<s<d_2$. By construction, this implies that
\[
\haus^1 \big(\partial^* \E \cap B(x,\rho)\big) \geq d_1 + d_2\,.
\]
As usual, we define the cluster $\E'$ coinciding with $\E$ outside of $B(x,\rho)$ and such that $\partial\E'\cap B(x,\rho)$ is given by the segment $ab$, so by Lemma~\ref{labello} we readily obtain
\[
h_{\rm min}\big(d_1+d_2 - |a-b|\big) \leq 2 \rho \omega(2\rho) + \Ceeb[2\Cgr \rho^\eta] (2 \Cgr \rho^\eta)^\beta\,.
\]
Keeping in mind that $|y-x|=r$ while $|a-x|=|b-x|=\rho>2r$, arguing exactly as in Lemma~\ref{aa2bp} we find a function $\tilde\xi$ which satisfies the Dini property and such that $|\angle ayb-\pi| \leq \tilde\xi(r)$. In addition, $\tilde\xi(r)\lesssim r^\gamma$ if $\eta\beta>1$ and $h$ is locally $\alpha$-H\"older, with $\gamma$ given by~(\ref{heregamma}). Moreover, Lemma~\ref{aa2bp} already gives that $|\angle axb-\pi|\leq \xi_1(\rho)\leq \xi_1(2C_2r)$.
\begin{figure}[th]
\begin{tikzpicture}[>=>>]
\draw[line width=1] (0,0) circle (2);
\draw[line width=1] (0,0) circle (4.5);
\fill (0,0) circle (2pt);
\draw (0,0) node[anchor=north east] {$x$};
\fill (-4.47,.5) circle (2pt);
\draw (-4.47,.5) node[anchor=south east] {$a$};
\fill (4.47,.5) circle (2pt);
\draw (4.47,.5) node[anchor=south west] {$b$};
\fill (-1.73,1) circle (2pt);
\draw (-1.73,1) node[anchor=south east] {$y$};
\draw[<->] (-1.53,1) -- (-1.53,.5);
\draw (-1.53,.75) node[anchor=west] {$\delta_y$};
\draw[<->] (.2,0) -- (.2,.5);
\draw (.2,.25) node[anchor=west] {$\delta_x$};
\draw (4,2.2) node[anchor=west] {$\partial B(x,\rho)$};
\draw (1.2,1.7) node[anchor=west] {$\partial B(x,r)$};
\draw[dashed] (-4.47,.5)--(4.47,.5);
\draw (-4.47,.5) -- (-1.73,1) -- (4.47,0.5);
\draw (-2.8,.75) node[anchor=south east] {$d_1$};
\draw (3.2,.6) node[anchor=south east] {$d_2$};
\end{tikzpicture}
\caption{The situation in Lemma~\ref{xi2} to define the map $\xi_2$.}\label{FigLem215}
\end{figure}
We are then in position to find a function $\xi_2$, satisfying the Dini property and with the same additional features as $\tilde\xi$, for which~(\ref{alignmenteveryc}) is true. Let us give the appropriate definition of $\xi_2$, with the aid of Figure~\ref{FigLem215} which depicts the situation. First of all, let us notice that
\begin{align*}
\big|\angle xab\big| = \frac{\big| \angle axb- \pi\big|}2 \leq \xi_1(2C_2 r)\,, &&
\big|\angle yab\big| \leq\big| \angle ayb- \pi\big| \leq \tilde\xi(r)\,.
\end{align*}
Calling then, as in the figure, $\delta_x$ and $\delta_y$ the distances of the points $x$ and $y$ from the segment $ab$, we have
\[
\delta_x = \rho  \sin\big(\big|\angle xab\big|\big) \leq 2 C_2 r \big|\angle xab\big| \leq 2 C_2 r \xi_1(2C_2 r)\,,
\]
and similarly
\[
\delta_y = d_1 \sin\big(\big|\angle yab\big|\big) \leq (\rho+r) \big|\angle yab\big| \leq 3 C_2 r \tilde\xi(r)\,.
\]
Now, keep in mind that $\zeta(y-x)$ and $\tau(x,r)$ are the directions of the vectors $y-x$ and $b-a$ respectively. Hence, by construction we have that
\[
|y-x| \sin\big(\big|\zeta(y-x)-\tau(x,r)\big|\big) \leq \delta_x + \delta_y\,,
\]
in particular the equality holds if $y$ and $x$ are on the opposite side of the segment $ab$, as in the figure, while otherwise the left hand side term equals $|\delta_x-\delta_y|\leq \delta_x+\delta_y$. Putting together the last three inequalities, we finally have
\[
|\zeta(y-x)- \tau(x,r)| \leq \frac \pi 2 \,  \sin\big(\big|\zeta(y-x)-\tau(x,r)\big|\big) \leq \frac{\pi(\delta_x+\delta_y)}{2r}
\leq 5C_2 \big(\xi_1(2C_2 r)+\tilde\xi(r)\big)\,,
\]
and so we obtain the searched function $\xi_2$ satisying~(\ref{alignmenteveryc}) by defining
\[
\tilde\xi(r) = 5C_2 \big(\xi_1(2C_2 r)+\tilde\xi(r)\big)\,.
\]
To conclude the proof, we only have to establish~(\ref{taucont}). By Lemma~\ref{aa2bp} we have
\[
\haus^1(\partial^* \E \cap B(x,\rho)) - 2\rho \leq \haus^1(\partial^* \E \cap B(x,\rho)) - |b-a| \leq \frac \rho 6\, \xi_1(\rho)^2\,,
\]
and, since $\haus^1(\partial^*\E\cap B(x,\rho))-\haus^1(\partial^*\E\cap B(x,r))\ge 2(\rho-r)$, this implies that
\begin{equation}\label{latefio}
\haus^1(\partial^*\E \cap B(x,r)) \leq 2r + \frac \rho 6\, \xi_1(\rho)^2\,.
\end{equation}
Let now $r'\in [r/2,r)$, and let $z\in\partial B(x,r')\cap\partial^*\E$. Let us call $y$ the point of $\partial B(x,r)\cap \partial^* \E$ which is closest to $z$. Then, calling for brevity $\theta=\zeta(z-x)-\zeta(y-x)$, since $r'\geq r/2$ it is 
\[
\haus^1(\partial^*\E \cap B(x,r))-2r \geq r\, \Big(\sqrt{1+\sin^2\theta}-1\Big) \geq \frac{\theta^2}6\, r\,.
\]
By~(\ref{latefio}) and~(\ref{alignmenteveryc}), we have then
\[
|\zeta(z-x)-\tau(x,r)| \leq \sqrt{2C_2}\xi_1(\rho) +\xi_2(r)\leq \sqrt{2C_2}\xi_1(2C_2 r) +\xi_2(r)\,.
\]
To conclude it is then enough to apply~(\ref{alignmenteveryc}) with $r'$ in place of $r$ and $z$ in place of $y$, finally finding
\[
|\tau(x,r')-\tau(x,r)| \leq |\zeta(z-x)-\tau(x,r') | + |\zeta(z-x)-\tau(x,r)|\leq \sqrt{2C_2}\xi_1(2C_2 r) +\xi_2(r) + \xi_2(r')\,,
\]
which is stronger than~(\ref{taucont}).
\end{proof}

\begin{cor}\label{corsss}
Let $B(\bar x,\bar r)$ be as in Proposition~\ref{step6} and let $x,\,y\in\partial^*\E$ be such that $r:=|y-x|<R_7/(2C_2)$ and $B(x,2C_2 r)\cup B(y,2C_2 r)\subseteq B(\bar x,\bar r)$. Then,
\[
|\tau(x,r) - \tau(y,r) | \leq 2\xi_2(r)\,.
\]
\end{cor}
\begin{proof}
It is possible to apply Lemma~\ref{xi2} both to $x$ and $y$. Then, (\ref{alignmenteveryc}) gives that the direction $\zeta(y-x)$ of the vector $y-x$ differs at most $\xi_2(r)$ from both $\tau(x,r)$ and $\tau(y,r)$. The thesis is then obvious.
\end{proof}

We are now in position to prove Proposition~\ref{step6}.

\begin{proof}[Proof (of Proposition~\ref{step6}).]
We let $\xi_1$ and $\xi_2$ be the functions defined in Lemmas~\ref{aa2bp} and~\ref{xi2}. For every $x\in \partial^*\E\cap B(\bar x,\bar r)$, it is possible to apply Lemma~\ref{xi2} for every $r$ small enough. For every such $r$, taking in account~(\ref{taucont}), by obvious induction we get that for every $n\in\N$ and every $r'\in [r/2^n,r/2^{n-1})$ one has
\[
|\tau(x,r)-\tau(x,r')| \leq 2 C_2 \sum_{j=0}^{n-1} \xi_1(2C_2 r/2^j) + 2\sum_{j=0}^{n-1} \xi_2(r/2^j)\,.
\]
Let us then define
\[
\xi_3(r) = 2C_2 \sum_{j=0}^{+\infty} \xi_1(2C_2 r/2^j) + 2\sum_{j=0}^{+\infty} \xi_2(r/2^j)\,.
\]
Notice that the series converges since the functions $\xi_1$ and $\xi_2$ have the Dini property. Moreover, if $\eta\beta>1$ and $h$ is locally $\alpha$-H\"older, then both $\xi_1$ and $\xi_2$ are bounded by a multiplicative constant (only depending on $h_{\rm min},\, \beta,\, \eta,\, \omega,\,\Ceeb$ and $\Cgr$) times $r^\gamma$, with $\gamma$ given by~(\ref{heregamma}). Hence, not only the series converges, but also $\xi_3(r)\leq K r^\gamma$ with a constant $K$ only depending on the data. \par
As a consequence, we obtain that $\tau(x,r')$ converges to a direction $\tau(x)\in\P^1$ for $r'\searrow 0$, and that $|\tau(x,r)-\tau(x)|\leq \xi_3(r)$. For every $x,\,y\in B(\bar x,\bar r)$ as in Corollary~\ref{corsss}, then, we deduce that, calling $r=|y-x|$, one has
\[
|\tau(x)-\tau(y)|\leq 2\xi_3(r) + 2\xi_2(r)\,.
\]
We can finally set $\xi(r)=2\xi_3(r) + 2\xi_2(r)$. Summarizing, we have shown that for \emph{every} $x\in\partial^*\E\cap B(\bar x,\bar r)$ the normal vector to $\partial^*\E$ at $x$ exists, and is orthogonal to $\tau(x)$. The above estimate, also keeping in mind~(\ref{only2colors}), ensures then that $\partial^*\E$ is a finite union of ${\rm C}^1$ curves.
\end{proof}

\begin{cor}[Single ${\rm C}^1$ curve]\label{step6+}
Let $B(\bar x,\bar r)\subseteq D$ be a ball as in Proposition~\ref{step6}, with the additional assumption that $\bar r< R_4$ and that $\#\, \big(\partial^* \E\cap \partial B(\bar x,\bar r)\big) =2$. Then, $\partial\E \cap B(\bar x,\bar r)$ is a ${\rm C}^1$ relatively closed curve, having both endpoints on $\partial B(\bar x,\bar r)$.
\end{cor}
\begin{proof}
Proposition~\ref{step6} already tells us that $\partial\E\cap B(\bar x,\bar r)$ is a finite union of pairwise disjoint relatively closed ${\rm C}^1$ curves. Every such curve cannot have an endpoint inside the ball $B(\bar x,\bar r)$, hence it is either a closed loop or a curve with both endpoints in $\partial B(\bar x,\bar r)$. On the other hand, a closed loop can be excluded since $\bar r<R_4$ thanks to Lemma~\ref{noisland} and Remark~\ref{noislandG}. Consequently, every curve has two endpoints in $\partial B(\bar x,\bar r)$ and, since there are only two points in $\partial^* \E\cap \partial B(\bar x,\bar r)$, we deduce that the curve is unique.
\end{proof}

\subsection{Conclusion}

In this short section we can now give the proof of Theorem~\ref{main}, which basically consists in putting together the technical results of the preceding sections.

\begin{proof}[Proof of Theorem~\ref{main}]
Let $\E\subseteq\R^2$ be a minimal cluster, and let us fix two large, closed balls $D^-\comp D\subseteq\R^2$.

Let $x\in D^-\cap\partial\E$ be any $3$-color point in the boundary of $\E$ (there are finitely many of these points by Proposition~\ref{dist>R6}). Then, by Lemma~\ref{9/10} there is some radius $r(x)<R_6$ such that the ball $B(x,r(x))$ is compactly contained in $D$, and its boundary contains exactly three points in $\partial^*\E$.\par

Let instead $x\in D^-\cap\partial\E$ be any point in the boundary of $\E$ which is not a $3$-color point. Then, by definition and by Lemma~\ref{3points} there is some radius $r(x)<R_6$ such that the ball $B(x,r(x))$ is compactly contained in $D$, has non-negligible intersection with at most $2$ sets $E_i$ with $0\leq i\leq m$, and its boundary contains exactly two points in $\partial^*\E$ (in principle there could be at most three such points, but as already noticed in the proof of Lemma~\ref{xi2} they are necessarily $2$).\par

By compactness, we can cover $D^-$ with finitely many balls $B_j=B(x_j,r_j)$, having radii $r_j<R_6$ and with the following property. For every $j$, either $x_j$ is a $3$-color point and $\partial B_j\cap\partial^* \E$ is done by three points, or $x_j$ is not a $3$-color point, the ball $B_j$ has non-negligible intersection with at most two different sets $E_i,\, 0\leq i\leq m$, and $\partial B_j\cap\partial^* \E$ is done by two points.

In the second case, by Proposition~\ref{step6} and Corollary~\ref{step6+} we know that $\partial^*\E\cap B_j$ is done by a ${\rm C^1}$ curve whose tangent vector $\tau$ satisfies the uniform estimate~(\ref{uniformC1}).

Let us then consider a ball $B_j$ centered at a $3$-color point, and let $a$ be one of the three points of $\partial B_j\cap\partial^* \E$. The point $a$ is not a $3$-color point, by Proposition~\ref{dist>R6}. Hence, a small ball centered in $a$ has non-negligible intersection with only two different sets $E_i$, so using again Lemma~\ref{3points} and Corollary~\ref{step6+} we obtain that $\partial^*\E$ is a uniformly ${\rm C}^1$ curve near $a$. The same of course holds near $b$ and $c$, the other two points of $\partial B_j\cap\partial^* \E$.\par

Therefore, there are three maximal (with respect to the inclusion) uniformly ${\rm C}^1$ curves in $B_j\cap\partial^* \E$, having one endpoint respectively in $a,\,b,\,c$. Since, as just observed, $\partial^*\E$ is a ${\rm C}^1$ curve around each point which is not a $3$-color point, by maximality the second endpoint of each of the three curves must be a $3$-color point inside $B_j$ (keep in mind that the curves have finite length since $\E$ is a minimal cluster). This means that the three curves meet at $x_j$, which is the only $3$-color point in $B_j$. Keeping in mind the uniform ${\rm C}^1$ property of the curves, given by~(\ref{uniformC1}), we deduce that the three curves arrive with a well-defined tangent vector at $x_j$. In other words, $\partial^*\E\cap B_j$ contains three ${\rm C}^1$ curves starting at $a,\,b$ and $c$ and meeting at $x_j$ arriving with three tangent vectors. By Lemma~\ref{noisland} and Remark~\ref{noislandG}, $\partial^*\E\cap B_j$ cannot have other points except these three curves. Finally, the fact that the tangent vectors at $x_j$ form three angles of $\frac 23\,\pi$ is an immediate consequence of Lemma~\ref{steiner}.\par

The fact that, if $\eta\beta>1$ and $h$ is locally $\alpha$-H\"older, then the arcs are not only ${\rm C}^1$ but also ${\rm C}^{1,\gamma}$ is already given by Proposition~\ref{step6}. The proof is then concluded.
\end{proof}

\section{Examples\label{s:examples}}

\subsection{Grushin plane\label{ssgr}}
An interesting example arising from sub-Riemannian geometry is the so-called Grushin plane, corresponding to $\R^2$ endowed with densities
\begin{equation}
h(x,\nu)=\sqrt{\nu_1^2+|x_1|^{2\alpha}\nu_2^2},\qquad g\equiv 1,\quad x=(x_1,x_2)\in\R^2,\ \nu\in\mathbb S^1,
\end{equation}
for $\alpha\geq 0$. In particular, for $\alpha=1$ this is a $2$-dimensional quotient of the Heisenberg group, setting of the celebrated Pansu's conjecture~\cite{P82,P82fr}.\par

In~\cite{MM}, the authors characterize isoperimetric sets in this framework. We also refer to~\cite{FM} for a multidimensional generalization of the isoperimetric problem, and to~\cite{F,FS} for a first approach to clustering problems in the Grushin plane. Note that in these references the Grushin perimeter is defined in a more general way via De Giorgi's definition allowing for non-Euclidean rectifiable sets. In this paper, we do not need to work at this level of generality since a suitable (non-smooth) change of coordinates (see~\cite[Proposition~2.3]{MM}) reduces the problem to the study of the densities
\begin{equation}\label{eq:grushin}
h\equiv1,
\qquad
g(x)=|(1+\alpha)x_1|^{-\frac{\alpha}{1+\alpha}},\quad x\in\mathbb{R}^2,
\end{equation}
for sets with locally finite Euclidean perimeter. 
Existence of minimal clusters for the densities in~\eqref{eq:grushin} is proved in the forthcoming paper~\cite{FPS}.

\begin{prop}\label{p:grushin}
Any minimal cluster $\mathcal E$ relative to the densities in~\eqref{eq:grushin} satisfies the Steiner property and the arcs of $\partial^*\mathcal E$ are ${\rm C}^{1,\gamma}$ with $\gamma=1/(2(\alpha+1))$.
\end{prop}

\begin{proof}
We show that the $\eta$-growth condition holds with $\eta=(\alpha+2)/(\alpha+1)$ and that any $m$-cluster $\mathcal E$ satisfies the $\eps-\eps^\beta$ property with $\beta=1$. The conclusion follows then by Theorem~\ref{main} (note that $h$ is the Euclidean density, hence regular).

We begin with the $\eta$-growth condition. For $x\in\R^2$ and $r>0$, let us set $Q(x,r)=[x_1-r,x_1+r]\times[x_2-r,x_2+r]$. In the following $C_\alpha>0$ will be a constant only depending on $\alpha$. Since $t\mapsto |t|^{-\frac{\alpha}{1+\alpha}}$ is decreasing for $t\in\R^+$, then
\[
|Q(x,r)|\leq |Q(0,r)|
=4(1+\alpha)^{-\frac{\alpha}{1+\alpha}}\, r\int_0^rx_1^{-\frac{\alpha}{1+\alpha}}\,dx_1
=C_\alpha\,r^{\frac{\alpha+2}{\alpha+1}}\,.
\]
The $\eta$-growth condition then holds with $\Cgr=C_\alpha$ and $R_\eta=1$.

We now pass to the $\eps-\eps^\beta$ property for $m$-clusters, starting from the case $m=1$. Let $E$ be a set of locally finite perimeter and finite Lebesgue measure. We fix $a,b\in\partial^*E$ such that $d=\min\{|a-b|,|a_1|,|b_1|\}>0$ and we let $B_1=B(a,d/4)$ and $B_2=B(b,d/4)$. Note that on $B_1\cup B_2$ we have $1/K<g<K$ for a constant $K>0$ only depending on the choice of these two balls, and set $R_\beta=d/8$. By construction, for every $x\in\R^2$ the ball $B(x,R_\beta)$ can intersect at most one among $B_1$ and $B_2$, so to get~(\ref{propeeb}) we can apply the standard Euclidean result in a ball among $B_1$ and $B_2$ not intersecting $B(x,R_\beta)$. To pass from the case $m=1$ to the case $m>1$, we can argue similarly as in the proof of~\cite[Theorem~2.9.14]{Mag}. More precisely, for any $1\leq i\leq m$, we can easily find finitely many indices $i_0,\, i_1,\, i_2,\, \dots\, , \, i_k\in \{0,\, 1,\, \dots\,,\,m\}$ such that $i_0=i,\, i_k=0$, and for every $0\leq j < k$ there is a point $a_j \in \partial^* E_{i_j}\cap \partial^* E_{i_{j+1}}$ not lying on the $x_2$-axis. Since the necessary points to fix are at most $m(m+1)/2$, we can apply the Euclidean result finitely many times obtaining~(\ref{propeeb}) for the special case when the vector $\eps\in\R^m$ has a single non-zero coordinate. And from this we obviously conclude also for a generic vector.
\end{proof}

\begin{rmk}
The isoperimetric set for the densities in~\eqref{eq:grushin} has a $C^{1,\frac{1}{\alpha+1}}$ regular boundary, as follows by~\cite{MM}. In particular, the regularity established in Proposition~\ref{p:grushin} is not sharp, at least for $m=1$. 
Moreover, in~\cite{FPS} we prove that minimal clusters in this framework exist and are bounded so that they are made by a \emph{finite} union of $C^{1,\frac{1}{2(\alpha+1)}}$ regular arcs. 
\end{rmk}

\subsection{Gaussian plane\label{ssgauss}}

The Gaussian plane is $\mathbb{R}^2$ with densities
\[
h(x)=g(x)=\frac{1}{2\pi}\, e^{-\frac{|x|^2}{2}},\quad x\in\mathbb{R}^2\,.
\]
The isoperimetric problem with these densities, which is very important also for its connections with Probability, is deeply studied since the pioneering works~\cite{ST,B}. Recently, the characterization of optimal $m$-clusters in this framework has been given in~\cite{MN}, where the problem is solved in the more general $n$-dimensional Gaussian space, for $2\leq m\leq n+1$. A simple application of our main result is the following.

\begin{prop}
Any minimal cluster $\mathcal E$ relative to the Gaussian densities satisfies the Steiner property and the arcs of $\partial^*\mathcal E$ are $C^{\infty}$.
\end{prop}

\begin{proof}
We first apply Theorem~\ref{main} to prove that the Steiner property holds with $C^{1,\frac{1}{2}}$ regularity. Indeed, the $\eta$-growth condition is easily verified with $\eta=2$ and the $\eps-\eps^\beta$ property for clusters holds with $\beta=1$ thanks to~\cite[Theorem~A]{PS19}.

To conclude the proof it is enough to observe that $h\equiv g$ is a smooth function on $\R^2$ and then the $C^\infty$ regularity of the arcs follows by a standard variational argument. 
\end{proof}

\section*{Acknowledgments}
The authors acknowledge the support of the GNAMPA--INdAM projects \textit{Problemi isoperimetrici in spazi Euclidei e non} (n.\ prot.\ UUFMBAZ-2019-000473 11-03-2019) and \textit{Problemi isoperimetrici con anisotropie} (n.\ prot.\ U-UFMBAZ-2020-000798 15-04-2020).
The first author also acknowledges the support received from the European Union's Horizon 2020 research and innovation programme under the \textit{Marie Sk\l odowska-Curie grant No 794592} and from the ANR-15-CE40-0018 project \textit{SRGI - Sub-Riemannian Geometry and Interactions}.
The third author also acknowledges the support of the \textit{ERC Starting Grant 676675 FLIRT – Fluid Flows and Irregular Transport}, and he has also received funding from the European Research Council (ERC) under the European Union’s Horizon 2020 research and innovation program (grant agreement No.\ 945655). Most of the work of this paper was performed while the third author was at the Department Mathematik und Informatik of the University of Basel (Switzerland). The authors would like to thank F.~Morgan for useful comments on a first draft of this paper.


\begin{thebibliography}{100}

\bibitem{ABCMP17} A. Alvino, F. Brock, F. Chiacchio, A. Mercaldo, M. R. Posteraro, Some isoperimetric inequalities on $\R^N$ with respect to weights $|x|^\alpha$. J. Math. Anal. Appl. {\bf 451} (2017), no. 1, 280--318.
\bibitem{ABCMP19} A. Alvino, F. Brock, F. Chiacchio, A. Mercaldo, M. R. Posteraro, The isoperimetric problem for a class of non-radial weights and applications, J. Differential Equations {\bf 267} (2019), no. 12, 6831--6871.
\bibitem{AFP} L. Ambrosio, N. Fusco, D. Pallara, Functions of Bounded Variation and Free Discontinuity Problems, Oxford University Press (2000).
\bibitem{B} C. Borell, The Brunn-Minkowski inequality in Gauss spaces. Invent. Math., {\bf 30} (1975), 207--216.
\bibitem{BBCLT16} W. Boyer, B. Brown, G. Chambers, A. Loving \& S. Tammen, Isoperimetric regions in $\mathbb{R}^n$ with density $r^p$, Anal. Geom. Metr. Spaces {\bf 4} (2016), no. 1, 236--265.
\bibitem{CRS12} X. Cabr\'e, X. Ros-Oton \& J. Serra, Euclidean balls solve some isoperimetric problems with nonradial weights, C. R. Math. Acad. Sci. Paris {\bf 350} (2012), no. 21--22, 945--947.
\bibitem{CMV10} A. Ca\~nete, M. Miranda \& D. Vittone, Some isoperimetric problems in planes with density, J. Geom. Anal. {\bf 20} (2010), no. 2, 243--290.
\bibitem{CGPRS20} E. Cinti, F. Glaudo, A. Pratelli, J. Serra \& X. Ros-Oton, Sharp quantitative stability for isoperimtric inequalities with homogeneous weights, preprint (2020).
\bibitem{CP} E. Cinti \& A. Pratelli, The $\eps-\eps^\beta$ property, the boundedness of isoperimetric sets in $\R^N$ with density, and some applications, J. Reine Angew. Math. (Crelle) {\bf 728} (2017), 65--103.
\bibitem{DFP} G. De Philippis, G. Franzina \& A. Pratelli, Existence of Isoperimetric sets with densities ``converging from below'' on $\R^N$, J. Geom. Anal. {\bf 27} (2017), no. 2, 1086--1105.
\bibitem{FP20} G. Franzina \& A. Pratelli, Non-existence of isoperimetric sets in the {E}uclidean space with vanishing densities, preprint (2020).
\bibitem{F} V. Franceschi, A minimal partition problem with trace constraint in the Grushin plane, Calc. Var. Partial Differential Equations {\bf 56} (2017), no. 4, 56--104.
\bibitem{FM} V. Franceschi \& R. Monti, Isoperimetric problem in $H$-type groups and Grushin spaces, Rev. Mat. Iberoam. {\bf 32} (2016), no. 4, 1227--1258.
\bibitem{FPS2} V. Franceschi, A. Pratelli \& G. Stefani, On the Steiner property for planar minimizing clusters. The anisotropic case, preprint (2020).
\bibitem{FPS} V. Franceschi, A. Pratelli \& G. Stefani, On the existence of planar minimizing clusters, preprint (2020).
\bibitem{FS} V. Franceschi \& G. Stefani, Symmetric double bubbles in the Grushin plane, ESAIM Control Optim. Calc. Var., to appear (2019).
\bibitem{G18} I. McGillivray, An isoperimetric inequality in the plane with a log-convex density, Ric. Mat. {\bf 67} (2018), no. 2, 817--874.
\bibitem{Gustin} W. Gustin, Boxing inequalities, J. Math. Mech. {\bf 9} (1960), 229--239.
\bibitem{Leonardi} G. P. Leonardi, Infiltrations in immiscible fluids systems. Proc. Roy. Soc. Edinburgh Sect. A {\bf131} (2001), no. 2, 425--436.
\bibitem{HMRR} M. Hutchings, F. Morgan, M. Ritor\'e \& A. Ros, Proof of the double bubble conjecture, Ann. of Math. {\bf 155} (2002), no. 2, 459--489.
\bibitem{Mag} F. Maggi, Sets of finite perimeter and geometric variational problems, Cambridge Studies in Advanced Mathematics {\bf 135}, Cambridge University Press, 2012.
\bibitem{MN} E. Milman \& J. Neeman, The Gaussian Double-Bubble and Multi-Bubble Conjectures, to appear on Ann. of Math. (2021).
\bibitem{MM} R. Monti \& D. Morbidelli, Isoperimetric inequality in the Grushin plane, J. Geom. Anal. {\bf 14} (2004), no. 2, 355--368.
\bibitem{Morganbook} F. Morgan, Geometric measure theory, a beginner's guide. Fifth edition, Elsevier/Academic Press, Amsterdam, 2016.
\bibitem{M94} F. Morgan, Soap bubbles in $\R^2$ and in surfaces. Pacific J. Math. {\bf 165} (1994), no. 2, 347--361.
\bibitem{MP13} F. Morgan \& A. Pratelli, Existence of isoperimetric regions in $\R^n$ with density, Ann. Global Anal. Geom. {\bf 43} (2013), no. 4, 331--365.
\bibitem{P82} P.~Pansu, An isoperimetric inequality on the Heisenberg group, Conference on differential geometry on homogeneous spaces (Turin,
1983). Rend.\ Sem.\ Mat.\ Univ.\ Politec.\ Torino, Special Issue (1983), 159--174.
\bibitem{P82fr} P. Pansu, Une in\'egalit\'e isop\'erim\'etrique sur le groupe de Heisenberg, C. R. Acad. Sci. Paris S\'er. I Math. \textbf{295} (1982), no. 2, 127--130.
\bibitem{PT0} E. Paolini \& A. Tamagnini, Minimal clusters of four planar regions with the same area. ESAIM Control Optim. Calc. Var. {\bf 24} (2018), no. 3, 1303--1331.
\bibitem{PT1} E. Paolini \& A. Tamagnini, Minimal cluster computation for four planar regions with the same area, Geometric Flow {\bf 3} (2018), no. 1, 90--96.
\bibitem{PT2} E. Paolini \& V.M. Tortorelli, The quadruple planar bubble enclosing equal areas is symmetric, to appear on Calc. Var. PDE (2020).
\bibitem{PS18} A. Pratelli \& G. Saracco, On the isoperimetric problem with double density, Nonlinear Analysis {\bf 177}, Part B (2018), 733--752.
\bibitem{PS19} A. Pratelli \& G. Saracco, The $\eps-\eps^\beta$ property for an isoperimetric problem with double density, and the regularity of isoperimetric sets, to appear on Adv. Nonlinear Stud. (2020).
\bibitem{PS2} A. Pratelli \& V. Scattaglia, The $\eps-\eps^\beta$ property for clusters with double density, preprint (2021).
\bibitem{Rei} B.W. Reichardt, Proof of the double bubble conjecture in $\R^n$, J. Geom. Anal. {\bf 18} (2008), no. 1, 172--191.
\bibitem{RHYS03} B.W. Reichardt, C. Heilmann,Y. Lai \& A. Spielman, 
Proof of the double bubble conjecture in $\R^4$ and certain higher dimensional cases.
Pacific J. Math. \textbf{208} (2003), no. 2, 347--366.
\bibitem{RCBM08} C. Rosales, A. Ca\~nete, V. Bayle \& F. Morgan, On the isoperimetric problem in Euclidean space with density, Calc. Var. Partial Differential Equations {\bf 31} (2008), no. 1, 27--46.
\bibitem{ST} V. N. Sudakov \& B. S. Cirel'son, Extremal properties of half-spaces for spherically invariant measures, Zap. NauVolcn. Sem. Leningrad. Otdel. Mat. Inst. Steklov. (LOMI), 41:14–24, 165, 1974. Problems in the theory of probability distributions, II.
\bibitem{Tam} I. Tamanini, Regularity results for almost-minimal oriented hypersurfaces in $\R^N$, Quaderni del Dipartimento di Matematica dell'Universit\`a del Salento (1984).
\bibitem{Tay} J. Taylor, The structure of singularities in soap-bubble-like and soap-film-like minimal surfaces, Ann. of Math. {\bf 103} (1976), no. 3, 489--539.
\bibitem{Volpert} A.I. Vol'pert, Spaces $BV$ and quasilinear equations, Math. USSR Sb. {\bf 17} (1967), 225--267.
\bibitem{W} W. Wichiramala, Proof of the planar triple bubble conjecture, J. Reine Angew. Math. {\bf 567} (2004), 1--49.
\end{thebibliography}
\end{document}